\documentclass{amsart}
\usepackage{amsmath}%
\usepackage{amsthm}
\usepackage{amsfonts}%
\usepackage{amssymb}%
\usepackage{graphicx}
\usepackage{tikz-cd}
\usetikzlibrary{cd}
\usepackage{bm}
\usepackage{comment}
\usepackage{bbm}
\usepackage{mathtools}

\usepackage{import}
\usepackage{xifthen}
\usepackage{pdfpages}
\usepackage{transparent}
\usepackage{xcolor}
\usepackage{stmaryrd}
\usepackage[shortlabels]{enumitem}

\usepackage{mathrsfs}

\DeclareFontFamily{OT1}{pzc}{}
\DeclareFontShape{OT1}{pzc}{m}{it}{<-> s * [1.10] pzcmi7t}{}
\DeclareMathAlphabet{\mathpzc}{OT1}{pzc}{m}{it}

\usepackage[hidelinks]{hyperref}
\newtheorem{theorem}{Theorem}[section]

\newtheorem{corollary}[theorem]{Corollary}

\newtheorem{definition}[theorem]{Definition}
\newtheorem{example}[theorem]{Example}

\newtheorem{lemma}[theorem]{Lemma}

\newtheorem{proposition}[theorem]{Proposition}
\newtheorem{remark}[theorem]{Remark}

\numberwithin{equation}{section}
\def\XXint#1#2#3{{\setbox0=\hbox{$#1{#2#3}{\int}$}
\vcenter{\hbox{$#2#3$}}\kern-.5\wd0}}

\newcommand{\R}{\mathbb{R}}

\newcommand{\N}{\mathbb{N}}
\newcommand{\Z}{\mathbb{Z}}

\newcommand{\Ha}{\mathcal{H}}
\newcommand{\leb}{\mathcal{L}}

\newcommand{\sgn}{\operatorname{sgn}}
\newcommand{\spt}{\operatorname{spt}}

\newcommand{\dist}{\operatorname{dist}}
\newcommand{\diam}{\operatorname{diam}}
\newcommand{\inte}{\operatorname{int}}

\newcommand{\Lip}{\operatorname{Lip}}
\newcommand{\LIP}{\operatorname{LIP}}
\newcommand{\im}{\operatorname{im}}

\newcommand{\ud}{\mathrm {d}}

\newcommand{\inv}{^{-1}}

\newcommand{\on}{\:\mbox{\rule{0.1ex}{1.2ex}\rule{1.1ex}{0.1ex}}\:}

\DeclareMathOperator{\bI}{\mathbf{I}}
\DeclareMathOperator{\biI}{\mathcal{I}}
\DeclareMathOperator{\st}{\textup{st}}
\DeclareMathOperator{\bM}{\mathbf{M}}
\DeclareMathOperator{\bN}{\mathbf{N}}

\DeclareMathOperator{\md}{\operatorname{md}}
\DeclareMathOperator{\mass}{\mathbf{M}}

\newcommand{\bb}[1]{\llbracket #1\rrbracket}
\usepackage{stmaryrd}
\newcommand{\norm}[1]{\lVert#1\rVert}

\usepackage[colorinlistoftodos]{todonotes}

\makeatletter
\newcommand*{\cone}{%
	{%
		\mathpalette\@coneOf{\times}%
	}%
}
\newcommand*{\@coneOf}[2]{%
	\sbox0{$\m@th#1\mathsf{#2}$}%
	\mathsf{#2}%
	\kern-\wd0 %
	\mkern2.00mu\relax
	\nonscript\mkern0.25mu\relax
	\mathsf{#2}%
}
\makeatother

\thanks{D.M. was supported by Swiss National Science Foundation grant 212867. E.S was supported by Finnish Academy Fellowship grant 355122. }
\address{Department of Mathematics\\ University of Fribourg\\  Chemin du Mus\'ee 23\\  1700 Fribourg, Switzerland}
\email{denis.marti@unifr.ch}

\address{Department of mathematics and statistics\\ University of Jyv\"askyl\"a, Mattilanniemi (MaD), PL35, 40014}
\email{elefterios.e.soultanis@jyu.fi}

\title{Characterization of metric spaces with a metric fundamental class}
\author{Denis Marti, Elefterios Soultanis}

\begin{document}

\begin{abstract}
We consider three conditions on metric manifolds with finite volume: (1) the existence of a metric fundamental class, (2) local index bounds for Lipschitz maps, and (3) Gromov--Hausdorff approximation with volume control by bi-Lipschitz manifolds. Condition (1) is known for metric manifolds satisfying the LLC condition by work of Basso--Marti--Wenger, while (3) is known for metric surfaces by work of Ntalampekos--Romney.

We prove that for metric manifolds with finite Nagata dimension, all three conditions are equivalent and that without assuming finite Nagata dimension, (1) implies (2) and (3) implies (1). As a corollary we obtain a generalization of the approximation result of Ntalampekos--Romney to metric manifolds of dimension $n\ge 2$, which have the LLC property and finite Nagata dimension.
\end{abstract}

\maketitle

\section{Introduction}

\subsection{Background}
A metric space homeomorphic to a smooth manifold is called a metric manifold. Interest in metric manifolds arises, for example, in the context of topological manifolds \cite{petersen-finiteness-metric,petersen-gh-limits}, in metric geometry \cite{gromov85}, as well as in quasiconformal theory (motivated by geometric group theory \cite{bonk-kle02}). It is known \cite{ferry-okun95} that geodesic metric manifolds can be approximated in the Gromov--Hausdorff metric with Riemannian metrics. Under additional geometric assumptions, namely Ahlfors regularity of the volume measure and having a linear local contractibility function (LLC condition), metric manifolds exhibit good analytic behaviour \cite{semmes-poincare}. These assumptions are also present in the work of Heinonen--Keith--Rickman--Sullivan \cite{Heinonen-Keith,heinonen-rickman02,heinonen-sullivan02} who considered the existence of local Euclidean bi-Lipschitz parametrizations of metric (cohomology) manifolds.

\medskip A central ingredient in the results in \cite{Heinonen-Keith,heinonen-rickman02,heinonen-sullivan02,semmes-poincare} is an analytic structure emerging from, and compatible with, the (geo)metric and topological assumptions. In this vein, Basso--Marti--Wenger \cite{basso2023geometric} recently proved the existence of a \emph{metric fundamental class} on metric manifolds with finite volume satisfying the LLC condition. The metric fundamental class is a generalisation of the fundamental class $\bb{M}$ given by the volume form of a closed, oriented, smooth manifold $M$, and is defined in terms of \emph{metric currents} (see Section \ref{sec: currents}).

\begin{definition}\label{def: metric fundamental class}
Let $X$ be a metric $n$-manifold with $\Ha^n(X)<\infty$ homeomorphic to a closed, oriented, smooth $n$-manifold $M$. A metric fundamental class of $X$ is an integral current $T\in \bI_n(X)$ with $\partial T=0$ such that
\begin{itemize}
    \item[(a)] there exists $C>0$ such that $\norm{T}\leq C \cdot \Ha^n$;
    \item[(b)] $\varphi_\# T = \deg(\varphi) \cdot \bb{M}$ for every Lipschitz map $\varphi \colon X \to M$.
\end{itemize}
\end{definition}

Here (a) requires that the volume measure associated to the metric fundamental class is compatible with the metric structure, while (b) requires compatibility with the topological structure of $X$. In particular, the metric fundamental class of a metric $n$-manifold generates the $n^\textrm{th}$ homology group via integral currents and is therefore unique, cf. \cite{basso2023geometric} and Remark \ref{rmk:multipl}.

\medskip The LLC condition in \cite{basso2023geometric} gives a stronger form of (a) which includes a lower bound $\|T\| \gtrsim \Ha^n$. It also provides control of the local index of Lipschitz maps, which is crucial in the construction and yields the topological compatibility. In dimension $n\ge 3$, it remains an open problem whether the LLC condition is necessary for the existence of a metric fundamental class. In two dimensions, however, it is redundant thanks to a polyhedral Gromov--Hausdorff approximation result of Romney--Ntalampekos \cite{nta-rom22,nta-rom23}. They proved that, given a metric surface $X$ with $\Ha^2(X)<\infty$, there is a sequence $P_j$ of polyhedral surfaces homeomorphic to $X$, and $\epsilon_j$-isometries $\varphi_j:X\to P_j$ with $\epsilon_j\to 0$ as $j\to\infty$ such that 
\begin{align}\label{eq:area-control}
\limsup_{j\to\infty}\Ha^2(\varphi_j(K))\lesssim \Ha^2(K),\quad K\subset X \textrm{ compact.}
\end{align}
For surfaces, this approximation implies a (weakly) quasiconformal parametrization, from which the existence of a metric fundamental class follows, see \cite{Meier-Wenger,nta-rom22,nta-rom23}. In the proof in \cite{nta-rom22}, planar topology yields a bound on the local index of Lipschitz maps, leading to the area control \eqref{eq:area-control} without the need of the LLC condition. In \cite{Meier-Wenger}, similar volume control is obtained using planar topology and finiteness of the Nagata dimension of geodesic metric surfaces. 

\medskip In this paper we consider metric manifolds in dimension $n\ge 2$ without assuming LLC, and show that the existence of a metric fundamental class is intimately linked to Gromov--Hausdorff approximation with the volume control \eqref{eq:area-control}, and to local index bounds for (bi-)Lipschitz maps. Indeed, under the additional assumption of finite Nagata dimension, both the volume control \eqref{eq:area-control} and the existence of a metric fundamental class are equivalent to local index bounds for (bi-)Lipschitz maps. Thus, for metric manifolds with finite volume and finite Nagata dimension, we obtain two characterizations for the existence of metric fundamental classes: one geometric and the other topological in nature. As a corollary, we generalize \cite[Theorem 1.1]{nta-rom22} to metric $n$-manifolds, $n\ge 2$, with finite Nagata dimension and LLC. In contrast to \cite{nta-rom22}, where the approximating surfaces are polyhedral, we instead construct approximating Lipschitz manifolds, and having finite Nagata dimension is needed here. However, our proof shows that having Gromov--Hausdorff approximations with volume control implies local index-bounds for Lipschitz maps without assuming finite Nagata dimension.

\subsection{Main result}
In the remainder of this paper, the term \emph{metric $n$-manifold} refers to metric spaces homeomorphic to a closed, oriented, connected smooth $n$-manifold. Let $X$ be a metric $n$-manifold. We refer to Section \ref{sec:prelis} for the definition of Nagata dimension $\dim_NX$ and local index $\iota(f,x)$ of a map $f:X\to \R^n$ at a point $x\in X$. A map $\varphi\colon Y \to Z$ between metric spaces is called an $\epsilon$-isometry if $|d(\varphi(x),\varphi(y)) - d(x,y)| \leq \epsilon$ for all $x,y \in Y$ and $N_\epsilon(\varphi(Y))=Z$.

\begin{theorem}\label{thme: main}
    Let $X$ be a metric $n$-manifold homeomorphic to a closed, oriented, connected smooth $n$-manifold $M$. Suppose $\Ha^n(X)<\infty$ and $\dim_NX<\infty$. Then the following properties are equivalent.
    \begin{itemize}
        \item[(1)] $X$ has a metric fundamental class.
        \item[(2)] There exists $D>0$ such that, if $f\in \LIP(X,\R^n)$ and $f|_E$ is bi-Lipschitz for some Borel $E\subset X$, then $|\iota(f,x)|\le D$ for $\Ha^n$-almost every $x\in E$;
        %\item There exists a constant $D>1$ such that if $E\subset X$ is $n$-rectifiable and $f\colon X \to \R^n$ Lipschitz then for almost all $y \in f(E)$ and every $x \in f^{-1}(y)\cap E$ we have $|\iota(f,x)| \leq D.$ 
        \item[(3)] There exist $C>0$, a sequence of metric spaces $X_k$ that are bi-Lipschitz homeomorphic to $M$, and $\epsilon_k$-isometric homeomorphisms $\varphi_k \colon X \to X_k$ with $\epsilon_k \to 0$ as $k \to 0$, such that
        $$\limsup_{k \to \infty} \Ha^n(\varphi_k(K)) \leq C \Ha^n(K)$$
        for each compact set $K\subset X$. (Here, $M$ is equipped with any Riemannian metric.)
    \end{itemize}
\end{theorem}
The constants in (1)-(3) depend quantitatively on each other and on the constants appearing in the definition of Nagata dimension. A sequence $X_k$ as in (3) converges in the Gromov-Hausdorff sense to $X$. For the implication $(2)\implies (3)$ the assumption that $X$ has finite Nagata dimension is not needed. 

\medskip We remark that not all metric manifolds with finite volume have finite Nagata-dimension, i.e. the assumption $\dim_NX<\infty$ is not vacuous. We give an example of a metric 2-manifold $X$ with $\Ha^2(X)<\infty$ and $\dim_NX=\infty$ in Appendix \ref{sec:example}.

\medskip Combining Theorem \ref{thme: main} with \cite{basso2023geometric} we obtain the following generalisation of \cite[Theorem 1.1]{nta-rom22}. Recall that a metric manifold $X$ is LLC, if there exist constants $C,R>0$ such that every ball $B(x,r)\subset X$ is contractible in $B(x,Cr)$ whenever $x\in X$ and $r\le R$.

\begin{corollary}\label{cor:rom-nta-generalisation}
    Let $X$ be an LLC metric $n$-manifold with $\Ha^n(X)<\infty$ and $\dim_NX<\infty$. Then there exist $C>0$, bi-Lipschitz manifolds $M_k$, and $\epsilon_k$-isometric homeomorphisms $\varphi_k:X\to M_k$ with $\epsilon_k\to 0$ and
    \[
    \limsup_{k \to \infty} \Ha^n(\varphi_k(K)) \leq C \Ha^n(K),\quad K\subset X\textrm{ compact}.
    \]
\end{corollary}
\begin{proof}
By \cite[Theorem 1.1]{basso2023geometric}, $X$ admits a metric fundamental class, and Theorem \ref{thme: main} now implies the existence of the claimed sequence.
\end{proof}

\medskip Whereas Theorem \ref{thme: main} and Corollary \ref{cor:rom-nta-generalisation} rely on having finite Nagata dimension, the following corollary does not depend on this assumption. It states that the Gromov--Hausdorff limit of a sequence of linearly locally contractible metric manifolds with volume control has a metric fundamental class. We remark that the limit need not be LLC. 

\begin{corollary}\label{cor:local-index-bounds}
Let $X_k$ and $X$ be metric $n$-manifolds, and let $\varphi_k:X_k\to X$ be continuous $\epsilon_k$-isometries ($\epsilon_k\to 0$) with $\deg(\varphi_k)\ne 0$ satisfying
\[
\limsup_{k\to\infty}\Ha^n(\varphi_k\inv(K)) \leq C \Ha^n(K),\quad K\subset X\textrm{ compact}.
\]
If each $X_k$ is LLC and $X$ has finite Hausdorff $n$-measure, then $X$ admits a metric fundamental class.
\end{corollary}
\begin{proof}
Corollary \ref{cor:local-index-bounds} follows from \cite[Theorem 1.1]{basso2023geometric} together with Theorem \ref{thm: lip-mfld-approx-implies-fund-current} since, by \cite[Lemma 5.1 and Proposition 5.2]{basso2023geometric} the constant in Definition \ref{def: metric fundamental class}(a) for the fundamental class of $X_k$ is independent of $k$.
\end{proof}

Observe that no uniform control on the LLC constants is needed in Corollary \ref{cor:local-index-bounds}. We do not know whether an arbitrary metric manifold with finite volume can be approximated by LLC metric manifolds satisfying the conditions in Corollary \ref{cor:local-index-bounds}.

\subsection{Outline} We prove the implication (1)$\implies$(2) in Section \ref{sec: degree-bound} (Theorem \ref{thm: fund-current-implies-degree-bound}). A key part of the argument is Lemma \ref{lemma: fund-current-pushforward-deg-local}, which shows that a metric fundamental class is compatible with the local topological degree. In particular, it implies the integrability of local index of a Lipschitz map. Together with a density argument this implies local index bounds. This implication does not use the assumption $\dim_NX<\infty$.

\medskip The implication (2)$\implies$(3) is proved in Theorem \ref{thm: lip-mfld-approx-only-degree} in Section \ref{sec: approx-lip-mfld} (with some preliminary results in Section \ref{sec:projection-lemma}). The proof uses a construction originating in \cite{lang-nagata} (see also \cite{basso2021undistorted,Meier-Wenger}), namely we construct finite simplicial complexes $\Sigma_\epsilon$ with a bound (depending on $\dim_NX$) on the dimension of the cells, which approximate the metric manifold $X$ in the sense that there are Lipschitz maps $X\to \Sigma_\epsilon$ and $\Sigma_\epsilon\to N_\epsilon^{l^\infty}(X)$ with controlled Lipschitz constants. The approximating Lipschitz manifolds are constructed using these simplicial complexes. In order to control their volumes, we use an argument similar to \cite[Section 5]{Meier-Wenger} using the local index bound (2). 

\medskip Finally, in Section \ref{sec:approx-lip-mfld-to-fund-class} we prove the implication (3)$\implies$(1) (as a consequence of the more general Theorem \ref{thm: lip-mfld-approx-implies-fund-current}). The argument here is direct and uses the compactness properties of integral currents. The volume control in (3) yields the necessary mass bounds to apply Ambrosio--Kirchheim's compactness theorem \cite[Theorem 5.2]{ambrosio-kirchheim-2000}, and also yields (a) in Definition \ref{def: metric fundamental class}. The stability of the topological degree then yields the degree identity (b) in Definition \ref{def: metric fundamental class}. This implication also does not need the assumption $\dim_NX<\infty$.

\section{Preliminaries}\label{sec:prelis}
\subsection{Metric notions}
Let $(X,d)$ be a metric space. We write $B(x,r)$ for the open ball with center $x \in X$ and radius $r> 0$. For $A \subset X$ and $r>0$, we denote the open $r$-neighborhood of $A$ by 
$$N_r(A) = \left\{x \in X \colon d(A,x) < r\right\}.$$

A map $f\colon X \to Y$ between metric spaces is said to be \textit{\(L\)-Lipschitz} if $d(f(x),f(y)) \leq L d(x,y)$ for all $x,y \in X$. We write $\LIP(f)$ for the smallest constant satisfying this inequality and call it the \textit{Lipschitz constant} of $f$. Furthermore, we denote by 
$$\Lip(f)(x) = \limsup_{y\to x}  \frac{d(f(x),f(y))}{d(x,y)}$$
the \textit{pointwise Lipschitz constant} of $f$ at $x\in X$. If $f$ is injective and the inverse $f^{-1}$ is $L$-Lipschitz as well, we say $f$ is \textit{$L$-bi-Lipschitz}. We write $\LIP(X,Y)$ for the set of all Lipschitz maps from $X$ to $Y$. If $Y= \R$ we abbreviate $\LIP(X,\R)=\LIP(X)$.
Given two maps \(f, g\colon X \to Y\), we write
\[
d(f, g)=\sup\big\{ d(f(x), g(x)) : x\in X\big\}
\]
for the uniform distance between \(f\) and \(g\). We will frequently make use of the following Lipschitz approximation result; see e.g. \cite[Lemma 2.1]{basso2023geometric}.

\begin{lemma}\label{lemma: lip-approx-easy}
Let \(f\colon X \to Y\) be a continuous map from a compact metric space \(X\) to a separable metric space \(Y\) that is an absolute Lipschitz neighborhood retract. Then, for every \(\epsilon>0\) there exists a Lipschitz map \(g \colon X \to Y\) with \(d(f, g)<\epsilon\). 
\end{lemma}

A metric space \(Y\) is an \textit{absolute Lipschitz neighborhood retract} if there exists \(C>0\) such that whenever \(Y\subset Z\) for some metric space \(Z\), then there exists an open neighborhood $U\subset Z$ of \(Y\) and a \(C\)-Lipschitz retraction \(R\colon U \to Y\). Important examples of absolute Lipschitz neighborhood retracts are closed Riemannian manifolds, see e.g. \cite[Theorem~3.1]{hohti-1993}. In particular, using this fact and Lemma \ref{lemma: lip-approx-easy}, it is not difficult to show that the following holds. If $M$ is a closed Riemannian manifold and $\epsilon>0$, then there exists $\delta>0$ such that whenever $f,g\colon X \to M$ are Lipschitz maps from a metric space $X$ into $M$ with $d(f,g) < \delta$, then there exists a Lipschitz homotopy $H$ between $f$ and $g$ satisfying $d(H(x,s),H(x,t)) \leq \epsilon$ for every $x \in X$ and every $s,t$.

\subsection{Rectifiability}
Let $X$ be a complete metric space. We denote by $\Ha^n$ the Hausdorff $n$-measure on $X$ and normalize it such that $\Ha^n$ equals the Lebesgue measure $\leb^n$ on Euclidean space. A $\Ha^n$-measurable set $E \subset X$ is said to be \textit{$n$-rectifiable}, if there exist countably many Borel sets $K_i \subset \R^n$ and Lipschitz maps $\varphi_i \colon K_i \to X$ such that
\begin{equation}\label{eq: def-rect}
    \Ha^n\left( E \setminus \bigcup_i \varphi_i(K_i)\right) = 0.
\end{equation}

It follows from \cite[Lemma 4]{kirchheim1994rectifiable} that for any $n$-rectifiable subset $E \subset X$ there exist countably many bi-Lipschitz maps $\varphi_i \colon K_i \subset \R^n \to X$ satisfying \eqref{eq: def-rect} such that $\varphi_i(K_i)$ are pairwise disjoint, cf. \cite[Lemma 4.1]{ambrosio-kirchheim-2000}. We call a $\Ha^n$-measurable set $P \subset X$ \textit{purely $n$-unrectifiable} if $\Ha^n(P \cap E)=0$ for every $n$-rectifiable set $E \subset X$. 

In Sections \ref{sec:projection-lemma} and \ref{sec: approx-lip-mfld} we will need the following theorem, which is one of the main tools in the proof of the perturbation theorem \cite[Theorem 1.1]{bate-2020}. Note that the lower density assumption in \cite[Theorem 1.1]{bate-2020} has recently been removed in \cite{bate24}. Theorem \ref{thme:weak-perturb-bate} below is a direct consequence of Theorem 4.9, Theorem 2.21 and Observation 4.2 in \cite{bate-2020} together with \cite{bate24}.

%\todo[inline]{Here we use David's result without the lower density assumption.}

\begin{theorem}\label{thme:weak-perturb-bate}
    Let $X$ be a complete metric space with finite Hausdorff $n$-measure and $f\colon X \to \R^n$ Lipschitz. Let $P \subset X$ be purely $n$-unrectifiable. Then, for any $\epsilon>0$ there exists a $(\LIP(f)+\epsilon)$-Lipschitz map $g \colon X \to \R^n$ satisfying
    \begin{enumerate}
        \item$ \norm{f(x)-g(x)}\leq \epsilon$ for each $x \in X$ and $f(x)=g(x)$ whenever $d(x,P)> \epsilon $;
        \item $\Ha^n(g(P)) \leq \epsilon.$
    \end{enumerate}
\end{theorem}

\subsection{Simplicial complexes}
We refer to \cite{Spanier-alg-coh} for more information on simplicial complexes. Throughout this article, a simplicial complex consists always of simplices that are isometric toa  rescaling of some standard simplex in $\R^n$ and have the same side length. Let $\Sigma$ be a simplicial complex. For $k\geq 0$, we denote by $\Sigma^k$ the \textit{$k$-skeleton} of $\Sigma$, that is, the set of all closed simplices of dimension $k$ in $\Sigma$. Furthermore, we write $\mathcal{F}(\Sigma)$ for the set of all closed simplices of $\Sigma$. We say $\Sigma$ is $n$-dimensional if $\Sigma^n\ne\varnothing$ and $\mathcal{F}(\Sigma)= \bigcup_{k=0}^n\Sigma^k$. For a simplex $\Delta \in \mathcal{F}(\Sigma)$, the \textit{open star} of $\Delta$ is defined by
$$\st \Delta = \big\{ \inte(\tau) \colon \tau \in \mathcal{F}(\Sigma) \textup{ and } \Delta \subset \tau\big\},$$
where $\inte \tau$ denotes the interior of $\tau$. We consider the following metric on a simplicial complex. Let $I$ be a finite set and let $\Sigma(I)$ be the simplicial complex defined as follows
\begin{equation}\label{eq: def-simplicial-complex-l2}
    \Sigma(I) = \left\{ x  \in l^2(I) \colon x_i\geq 0 \textup{ and } \sum_i x_i =1 \right\}.
\end{equation}
Notice that, up to rescaling, any finite simplicial complex can be realized as a subcomplex of $\Sigma(I)$ with $I = \Sigma^0$. The \textit{$l^2$-metric} on $\Sigma$ is the metric induced by the $l^2$-norm on $l^2(I)$ and is denoted by $|\cdot|_{l^2}$. The diameter of a simplex in $\Sigma$ with respect to the $l^2$-metric is equal to $\sqrt{2} s$, where $s$ is the side-length of the simplices in $\Sigma$.
%In this case, each simplex is isometric to the standard simplex of the same side length and same dimension.
We conclude the section with the following lemma whose proof is analogous to \cite[Lemma 5.2]{basso2023lipschitz}.

\begin{lemma}\label{lemma: close-points-neighb-simplicies}
    Let $\Sigma$ be a simplicial complex equipped with the $l^2$-metric. If $\Delta,\Delta' \in \mathcal{F}(\Sigma)$ are two simplices with $\Delta\cap \Delta' \neq \emptyset$, then for all $x \in \Delta$ and $y \in \Delta'$ there exists $z \in \Delta\cap \Delta'$ satisfying
    $$|x-z|_{l^2}+|z-y|_{l^2} \leq 4 \sqrt{n} |x-y|_{l^2},$$
    where $n= \min\left\{ \dim \Delta, \dim \Delta'\right\}$.
\end{lemma}

\subsection{Nagata dimension}
A covering $\{U_i\}_i$ of a metric space \(X\) is said to have \textit{\(s\)-multiplicity} at most \(n\) if every subset of \(X\) with diameter less than \(s\) intersect at most \(n\) members of $\{U_i\}_i$.

\begin{definition}\label{def: nagata-dim}
The Nagata dimension $\dim_N X$ of a metric space is the smallest non-negative integer \(n\) such that there exists a constant $c>0$ with the following property: for every \(s>0\) there exists a covering with \(s\)-multiplicity at most \(n+1\) such that each member \(B\) of the covering satisfies \(\diam B \leq cs\).
\end{definition}

Every doubling metric space has finite Nagata dimension and the dimension depends only on the doubling constant. Furthermore, the Nagata dimension is always at least the topological dimension and can be seen as a quantitative version of it, which is more suitable for Lipschitz analysis. Indeed, while the topological dimension is an important tool for constructing continuous extensions, the Nagata dimension has been applied to Lipschitz extension problems; see \cite{lang-nagata}.  In both cases, an approximation of the space by a simplicial complex plays a crucial role. These simplicial complex are obtained as the nerve of the cover given by the corresponding dimension. See \cite{petersen-gh-limits} and \cite{semmes-poincare} for the concepts mentioned for the topological dimension. In contrast to the topological dimension, the approximation obtained by the Nagata dimension have a uniform behavior on all scales. The following version of such an approximation by a simplicial complex will be used in Section \ref{sec: approx-lip-mfld}.

\begin{theorem}\label{thm:simplicial-factorization}
     Let $X \subset l^\infty$ be compact with Nagata dimension $\leq n$. Then there exists a constant $C>0$ only depending on the data of $X$ such that for every $\epsilon>0$ there is a finite simplicial complex $\Sigma$ equipped with the $l^2$-metric and Lipschitz maps $\varphi\colon X \to \Sigma$, $\psi\colon \Sigma \to l^\infty$ with the following properties
     \begin{enumerate}
         \item $\varphi$ is $C$-Lipschitz and $\psi$ is $C$-Lipschitz on every simplex $\Delta$ in $\Sigma$;
         \item $\Sigma$ has dimension $\leq n$ and every simplex in $\Sigma$ is a Euclidean simplex of side length $\epsilon$;
         \item $\textup{Hull}(\varphi(X)) = \Sigma$ and $d(x,\psi(\varphi(x))) \leq C \epsilon$ for all $x \in X$.
    \end{enumerate}
\end{theorem}

Here, $\textup{Hull}(A)$ denotes the \textit{hull} of $A$ in $\Sigma$ and is the smallest subcomplex of $\Sigma$ containing $A$.

\begin{proof}
    Since $l^\infty$ is an injective metric space the theorem follows directly from \cite[Proposition 6.1]{basso2021undistorted} and a simple rescaling argument.
\end{proof}

\subsection{Orientation and degree}
In the following we introduce the notion of degree and local index of a continuous map between manifolds. For more details we refer to \cite{Heinonen-Keith}, \cite{heinonen-rickman02} and \cite{heinonen-sullivan02}. Let $X$ be a topological $n$-manifold. We always assume $X$ to be connected. For $U \subset X$ we denote by $H^n_c(U)$ the Alexander-Spanier cohomology groups of $U$, with compact supports and coefficients in $\Z$, see e.g \cite{bredon2012sheaf}. An open and connected subset $D \subset X$ is called a \textit{domain}. We say $X$ is \textit{orientable} if $H^n_c(X)\cong	\Z$. A generator $g_X$ of $H^n_c(X)$ is called an orientation. If such an orientation exists, then it induces an orientation $g_D$ on every domain $D \subset X$ by the standard isomorphism $H_c^n(D) \to H_c^n(X)$. An orientable manifold with such choice of an orientation is said to be oriented.

\medskip Let $f\colon X \to Y$ be a continuous map between two oriented topological $n$-manifolds. If $D\subset X$ is a relatively compact domain, then for each connected component $V$ of $Y\setminus f(\partial D)$ there exists a homomorphism $H_c^n(V) \to H_c^n(D)$. This map sends the generator of $H_c^n(V)$ to an integer multiple of the generator of $H_c^n(D)$. We denote this integer by $\mu(f,D,y)$ and call it the \textit{local degree} of $f$ at $y \in V$ with respect to $D$. Notice that the definition does not depend on the choice of $y \in V$. It follows that the local degree 
\begin{align*}
\mu(f,D,\cdot) \colon &Y\setminus f(\partial D) \to \Z \\
        & y \mapsto \mu(f,D,y)
\end{align*}
is locally constant. Let $y \in Y \setminus f(\partial D)$ and suppose that there exists a domain $U \subset D$ such that $U \cap f^{-1}(y)$ consists only of a single point $x$. Then, we denote the value of $\mu(f,U,y)$ by $\iota(f,x)$ and call it the \textit{local index} of $f$ at $x$. If $X$ and $Y$ are closed, oriented manifolds, then $\mu(f,X,y)$ has the same value for each $y \in Y$. In this case, we simply write $\deg(f)$ and call it the \textit{degree} of $f$. We will make use of the following basic properties of the degree; see e.g. \cite{rado-degree}.

\begin{lemma}\label{lemma: basic-properties-degree}
    Let $f\colon X \to Y$ be a continuous map between two oriented $n$-manifolds and $D \subset X$ a relatively compact domain.
    \begin{enumerate}
        \item If $y \in Y \setminus f(\overline{D})$ then $\mu(f,D,y)=0.$
        \item If $y \in Y\setminus f(\partial D)$ and $D_1,\dots,D_k \subset D$ are pairwise disjoint domains such that $D \cap f^{-1}(y)\subset \bigcup_i^k D_i$ then
        $$\mu(f,D,y) = \sum_i^k \mu(f,D_i,y).$$
        \item If $y \in Y\setminus f(\partial D)$ and $g \colon X \to Y$ is continuous such that there exists a homotopy $H$ between $f$ and $g$ satisfying $y \neq H(x,t)$ for all $x \in \partial D$ and each $t$. Then $\mu(f,D,y) = \mu(g,D,y)$.
        \item If $T \colon \R^n \to \R^n$ is a linear bijection then $\mu(T,U, y) = \sgn (\det T)$ for every domain $U \subset \R^n$ and each $y \in T(U)$.
    \end{enumerate}
\end{lemma}

\begin{remark}\label{rmk:local-index-exists}
Let $X$ be a metric manifold with $\Ha^n(X)<\infty$, and let $f\in \LIP(X,\R^n)$. By the co-area inequality \cite[2.10.25]{federer-gmt} we have that $\#f\inv(y)<\infty$ for almost every $y\in \R^n$, and for such $y$ it is easy to see that $\iota(f,x)=\mu(f,D_x,y)$ exists for each $x\in f\inv(y)$. In particular if $f|_E$ is bi-Lipschitz, then $\iota(f,x)$ exists for $\Ha^n$-almost every $x\in E$.
\end{remark}
We record here a result which rephrases Theorem \ref{thme: main}(2) in slightly different terms that will be useful in the sequel.
\begin{lemma}\label{lemma: equiv-def-degree-bound}
Let $X$ be a metric $n$-manifold with $\Ha^n(X)<\infty$, and let $D\ge 1$. The following are equivalent.
\begin{itemize}
    \item[(i)] For all $f\in \LIP(X,\R^n)$ and $E\subset X$ with $f|_E$ bi-Lipschitz we have that $|\iota(f,x)|\le D$ for $\Ha^n$-almost every $x\in E$;
    \item[(ii)] For all $f\in \LIP(X,\R^n)$ and $n$-rectifiable $E\subset X$ we have that, for almost every $y\in \R^n$, $\#(f\inv(y)\cap E)<\infty$ and $|\iota(f,x)|\le D$ for each $x\in f\inv(y)\cap E$.
\end{itemize}
\end{lemma}
\begin{proof}
Clearly (ii)$\implies$(i) since, if $f\in \LIP(X,\R^n)$ and $E\subset X$ are such that $f|_E$ is bi-Lipschitz, then $\overline E$ is rectifiable. 

\medskip It therefore remains to prove the implication (i)$\implies$(ii). By Remark \ref{rmk:local-index-exists}, for almost every $y\in \R^n$ we have that $\iota(f,x)$ exists for each $x\in f\inv(y)$. Using \cite[Corollary 8]{kirchheim1994rectifiable} and \cite[Lemma 4]{kirchheim1994rectifiable} we find pairwise disjoint sets $E_i\subset E$ such that $f|_{E_i}$ is bi-Lipschitz and $\leb^n(f(N))=0$, where $\displaystyle N=E\setminus \bigcup_iE_i$. Thus for almost every $y\in \R^n$ we have that $\displaystyle f\inv(y)\cap E\subset \bigcup_iE_i$, from which (i) implies that $|\iota(f,x)|\le D$ for each $x\in f\inv(y)\cap E$, for almost every $y\in \R^n$.
\end{proof}

\section{Metric currents}\label{sec: currents}
We recall the basic definitions and facts about the theory of metric currents and refer to \cite{ambrosio-kirchheim-2000,lang-local} for more information.
Let $X$ be a complete metric space. We write $\mathcal{D}^k(X)$ for the set of all $(k+1)$-tuples $(f,\pi_1,\dots,\pi_k) \in \LIP(X,\R)^{k+1}$ of Lipschitz functions, with \(f\) bounded.

\begin{definition}\label{def:metric-current}
A multilinear map \(T\colon \mathcal{D}^k(X)\to \R\) is called metric \(k\)-current (of finite mass) if the following holds:
\begin{enumerate}
    \item(continuity) If \(\pi_i^{j}\) converges pointwise to \(\pi_i\) for every \(i=1, \ldots, k\), and \(\LIP \pi_i^{j}< C\) for some uniform constant \(C>0\), then 
    \[
    T(f, \pi_1^{j}, \ldots, \pi_k^{j})\to T(f, \pi_1, \ldots, \pi_k)
    \] as \(j\to \infty\);
    \item(locality) \(T(f, \pi_1, \ldots, \pi_k)=0\) if for some \(i\in\{1, \ldots, k\}\) the function \(\pi_i\) is constant on \(\{ x\in X : f(x) \neq 0\}\);
    \item(finite mass) There exists a finite Borel measure \(\mu\) on \(X\) such that
    \begin{equation}\label{eq:mass-inequality}
    |T(f, \pi_1, \ldots, \pi_k)| \leq \prod_{i=1}^k \LIP \pi_i \int_{X} |f(x)| \, d\mu(x)
    \end{equation}
    for all \((f, \pi_1, \ldots, \pi_k)\in \mathcal{D}^k(X)\).
    \end{enumerate}
\end{definition}

The minimal measure \(\mu\) satisfying \eqref{eq:mass-inequality} is called the \textit{mass measure} of \(T\) and we denote it by \(\norm{T}\). The \textit{support} of \(T\) is defined as
\[
\spt T=\big\{ x\in X : \norm{T}(B(x, r)) >0 \text{ for all } r>0\big\}.
\] 
We let $\bM_k(X)$ be the vector space of all metric $k$-currents on $X$ and $\mass(T) = \norm{T}(X)$ the mass of $\bM_k(X)$. It follows from a density argument that each metric $k$-current $T$ can be extended uniquely to $L^1(\norm{T}) \times \LIP^k(X)$. For a Borel set $B \subset X$, we write $T \on B$ for the restriction of $T$ to $B$ defined by 
\[
(T\on B)(f, \pi_1\ldots, \pi_k)=T(\mathbbm{1}_B \cdot f, \pi_1, \ldots, \pi_k).
\]
It is a well-defined metric \(k\)-current and satisfies \(\norm{T\on B}=\norm{T}\on B\). We say a sequence of metric $k$-currents $T_i\in \bM_k(X)$ converges weakly to $T \in \bM_k(X)$ if
$$\lim_{i \to \infty}T_i(f,\pi_1,\dots,\pi_k) = T(f,\pi_1,\dots,\pi_k)$$
for all $(f,\pi_1,\dots, \pi_{k})\in \mathcal{D}^k(X)$ and we write $T_i \rightharpoonup T$. For $T \in \bM_k(X)$ the \textit{boundary} of $T$ is the multilinear map on $\mathcal{D}^{k-1}(X)$ defined by
$$\partial T (f,\pi_1,\dots,\pi_{k-1}) = T(1,f,\pi_1,\dots,\pi_{k-1})$$
for all $(f,\pi_1,\dots,\pi_{k-1})\in \mathcal{D}^{k-1}(X)$. The boundary $\partial T$ satisfies the continuity and locality property. Therefore, if $\partial T$ has finite mass, then it is a metric $(k-1)$-current. In this case, we call $T$ a normal $k$-current. We write $\bN_k(X)$ for the set of all normal $k$-currents in $X$ and define $\bN(T) = \bM(T) + \bM(\partial T)$.

Let $\varphi\colon X \to Y$ be a Lipschitz map and $T \in \bM_k(X)$, the \textit{pushforward} of $T$ under $\varphi$ is the metric $k$-current in $Y$ given by
$$\varphi_\# T(f,\pi_1,\dots,\pi_k) = T(f\circ \varphi,\pi_1\circ \varphi,\dots,\pi_k\circ\varphi)$$
for all $(f,\pi_1,\dots, \pi_{k})\in \mathcal{D}^k(Y)$. The following properties can be verified directly from the definitions: $(\varphi_\#T)\on B = \varphi_\#(T\on\varphi^{-1}(B))$, the support of $\varphi_\#T$ is contained in the closure of $\varphi(\spt T)$ and $\bM(\varphi_\# T) \leq \LIP(\varphi)^k \bM(T)$.

\subsection{Integer rectifiable and integral currents}
In Euclidean space, the standard example of a metric $k$-current is given by integration. Indeed, let $\theta \in L^1(\R^k)$, then it defines a metric $k$-current $\bb{\theta} \in \bM_k(\R^k)$ as follows
$$\bb{\theta}(f,\pi_1,\dots,\pi_k) = \int_{\R^k} \theta f \det(D \pi) \; d\leb^k$$
for all $(f,\pi) = (f,\pi_1,\dots,\pi_k) \in \mathcal{D}^k(\R^k).$ These currents are used as building blocks for integer rectifiable currents in metric spaces in analogy to the definition of rectifiable sets. Let $X$ be a complete metric space. We say a $k$-current $T \in \bM_k(X)$ is integer rectifiable if there exist countably many compact sets $K_i \subset \R^k$ and $\theta_i \in L^1(\R^k,\Z)$ with $\spt \theta_i \subset K_i$ and bi-Lipschitz maps $\varphi_i \colon K_i \to X$ such that
\[
T=\sum_{i\in \N} \varphi_{i\#}\bb{\theta_i} \quad \text{ and } \quad \mass(T)=\sum_{i\in \N} \mass(\varphi_{i\#}\bb{\theta_i}).
\]
An integer rectifiable current that is also a normal current is called integral current. We write $\biI_k(X)$ and $\bI_k(X)$ for the space of all integer rectifiable currents and integral $k$-currents in $X$, respectively. Finally, an integral current $T$ without boundary $\partial T = 0$ is called an integral cycle. We record the following two technical lemmas for use later on.

\begin{lemma}\label{lemma: integer-currents-pushfwrd-lip-constant}
    Let $E\subset X$ be Borel and $f\colon X \to \R^k$ Lipschitz. If $f$ restricted to $E$ is $L$-Lipschitz, then
    $$\norm{f_\# T}(B) \leq L^k \norm{T}(f^{-1}(B))$$
    for every $T\in \bM_k(X)$ concentrated on $E$ and every Borel set $B \subset \R^k$.
\end{lemma}

\begin{proof}
Since $\|T\|(X\setminus E)=0$, we may regard $T\in \bM_k(X)$ as a current $T_E$ in $\overline E\subset X$ (cf. \cite[Proposition 3.3]{lang-local}). Note that $f\colon\overline E\to\R^k$ is $L$-Lipschitz and that $\|T_E\|\le \|T\|$. For any $(\pi_0,\ldots, \pi_k)\in \mathcal D_k(\R^k)$ with $\LIP(\pi_j)\le 1$ for $j=1,\ldots,k$, we have that 
\begin{align*}
f_\#T(\pi_0,\ldots,\pi_k)&=T(\pi_0\circ f,\ldots,\pi_k\circ f)= T_E(\pi_0\circ f|_{\overline E},\ldots, \pi_k\circ f|_{\overline E})\\
& \le \LIP(\pi_1\circ f|_{\overline E})\cdots \LIP(\pi_k\circ f|_{\overline E})\int_{\overline E}|\pi_0\circ f|_{\overline E}|\ud\|T_E\|\\
& \le L^k\int |\pi_0\circ f|\ud\|T\|,
\end{align*}
which implies that $\|f_\# T\|\le L^kf_\#\|T\|$. The claim follows.
\end{proof}

\begin{lemma}\label{lem:rect-criterion}
Let $X$ be a metric space with $\Ha^n(X)<\infty$ and $T\in \bN_k(X)$. Suppose that $f_\#(T\on U)\in \biI_k(\R^k)$ for every $f\in \LIP(X,\R^k)$ and every open set $U\subset X$ with $\Ha^{n-1}(\partial U)<\infty$. Then $T\in \bI_k(X)$.
\end{lemma}
\begin{proof}
First, let $U\subset X$ be an arbitrary open set, and define $U_\delta=\{x\in U:\ \dist(x,X\setminus U)>\delta\}$. Then $\partial U_\delta\subset \{x: \dist(x,X\setminus U)=\delta\}$ and thus, by the area inequality, for almost every $\delta>0$ we have that $\Ha^{n-1}(\partial U_\delta)<\infty$. Choosing a decreasing sequence $\delta_j\downarrow 0$ with $\Ha^{n-1}(\partial U_{\delta_j})<\infty$, it follows from the dominated convergence theorem that $\|T\|(U\setminus U_{\delta_j})\to 0$ as $j\to\infty$. Thus $\bM(f_\#(T\on U)-f_\#(T\on U_{\delta_j}))\to 0$ as $j\to\infty$ for every $f\in \LIP(X,\R^k)$. Since $f_\#(T\on U_{\delta_j})\in \biI_k(\R^k)$ for all $j$ by assumption, it follows that $f_\#(T\on U)\in \biI_k(\R^k)$.

\medskip Similarly, given an arbitrary Borel set $B\subset X$ and $f\in \LIP(X,\R^k)$, we may find nested open sets $U_j$ containing $B$ so that $\|T\|(U_j\setminus B)\to 0$ as $j\to \infty$ using the Borel regularity of $\|T\|$ (see also \cite[Lemma 5.5]{ambrosio-kirchheim-2000}). Arguing as above we obtain that $f_\#(T\on B)\in \biI_k(\R^k)$. The claim now follows from \cite[Theorem 8.8]{ambrosio-kirchheim-2000}. 
\end{proof}

\begin{example}\label{ex:bilip}
The volume form of the closed, oriented, smooth $n$-manifold $M$ gives rise to the fundamental class $\bb M$ of $M$, which can be regarded as an integral cycle. This cycle is given by integrating Lipschitz differential forms \(f d\pi_1 \wedge \dotsm \wedge d\pi_n\) on \(M\). More precisely,
\[
\bb{M}(f,\pi) = \int_Mf\det(D\pi)\,d\hspace{-0.14em}\Ha^n
\]
for all \((f, \pi)\in \mathcal{D}^n(M)\). It follows from Stokes theorem that this defines an integral \(n\)-cycle, and satisfies $\bM(\bb M)=\Ha^n(M)$. Moreover, if $\varphi:M\to X$ is a bi-Lipschitz homeomorphism, then $\bb X:=\varphi_\#\bb M$ defines an integral cycle on $X$. Arguing as in \cite[Lemma 5.1]{basso2023geometric} we furthermore obtain that $\bM(\bb X)\le n^{n/2}\Ha^n(X)$.
\end{example}

\subsection{The homotopy formula}
Let $[0,1]\times X$ be equipped with the Euclidean product metric. Given a function $f\colon [0,1]\times X\to \R$, we define $f_t\colon X\to \R$ by $f_t(x):= f(t,x)$ for each $t\in[0,1]$. Furthermore, for $k\geq 0$ and $T\in\bI_k(X)$, let $\bb{t}\times T$ be the integral current in $\bI_k([0,1]\times X)$ given by
$$\bb{t}\times T(f,\pi_1,\dots, \pi_k):= T(f_t, \pi_{1t},\dots,\pi_{kt}).$$

Finally,  we define the multi-linear functional $\bb{0,1}\times T$ on $\mathcal{D}^{k+1}([0,1]\times X)$ by
\begin{align*}
    (f,\pi_1,&\dots, \pi_{k+1})\mapsto 
    \\
    &\sum_{i=1}^{k+1}\int_0^1(-1)^{k+1}T\Bigl(f_t\frac{\partial \pi_{it}}{\partial t},\pi_{1t},\dots, \pi_{(i-1)t}, \pi_{(i+1)t},\dots, \pi_{(k+1)t}\Bigl)\,dt.
\end{align*}
It has the following well-known properties; see e.g. \cite[Proposition~3.3]{basso2023geometric}.

\begin{proposition}\label{prop: homotopy-formula}
    Suppose that $T\in\bI_k(X)$. Then  $\bb{0,1}\times T\in \bI_{k+1}([0,1]\times X)$ and $$\partial(\bb{0,1}\times T) + \bb{0,1}\times\partial T = \bb{1}\times T - \bb{0}\times T.$$ For $k=0$ the second term on the left-hand side is zero.
\end{proposition}

The following easy corollary will be used in Section \ref{sec: degree-bound}.

\begin{lemma}\label{lemma: currents-homotopy}
    Let $f,g\colon X \to Y$ be Lipschitz maps between two complete metric spaces with $\Ha^{k+1}(Y)=0$. Suppose that there exists a Lipschitz homotopy $H\colon [0,1]\times X\to Y$ between $f$ and $g$ satisfying $d(H(x,0),H(x,t)) \leq \epsilon$ for all $x \in X$ and each $t \in [0,1]$. Then, 
    $$(f_\#(T\on U))\on Y\setminus N_\epsilon(f(\partial U)) = (g_\#(T \on U))\on Y\setminus N_\epsilon(f(\partial U))$$
    for every $T \in \bI_k(X)$ with $\partial T=0$ and any $U \subset X$ open with $\Ha^k(\partial U)=0$.
\end{lemma}

\begin{proof}
    It follows from the slicing inequality \cite[Theorem 5.6 and 5.7]{ambrosio-kirchheim-2000} that for almost all $r>0$ we have $T_r = T\on N_r(U) \in \bI_k(X)$ and $\spt(\partial T_r) \subset \{x \in X \colon d(x,U)=r\}$. Let $Q_r = H_\#([0,1] \times T_r) \in \bI_{k+1}(Y)$. By Proposition \ref{prop: homotopy-formula} 
    $$\partial Q_r = f_\# T_r - g_\# T_r- H_\#([0,1]\times \partial T_r).$$
    Notice that $\spt(H_\#([0,1]\times \partial T_r)) \subset N_\epsilon(f(\spt(\partial T_r)))  \subset N_{\LIP(f) r +\epsilon}(f(\partial U))$. Since $\Ha^{k+1}(Y)=0$ we have $Q=0$. Therefore,
    $$(f_\# T_r) \on Y \setminus N_{\LIP(f) r +\epsilon}(f(\partial U)) = (g_\# T_r) \on Y \setminus  N_{\LIP(f) r +\epsilon}(f(\partial U)).$$
    Moreover, $f_\# T_r \rightharpoonup f_\# (T\on U)$ and $g_\# T_r \rightharpoonup g_\# (T\on U)$ because $\Ha^k(\partial U)=0$. We conclude that
    $$(f_\# (T \on U)) \on Y \setminus N_{\LIP(f) r +\epsilon}(f(\partial U)) = (g_\# (T \on U)) \on Y \setminus  N_{\LIP(f) r +\epsilon}(f(\partial U))$$
    for almost all $r>0$. The statement follows by letting $r \to 0$.
\end{proof}

\section{Bounds for the local index}\label{sec: degree-bound}
Throughout this section, $X$ will be a metric $n$-manifold with $\Ha^n(X)<\infty$ homeomorphic to a closed, oriented, smooth $n$-manifold $M$. The following theorem proves the implication (1)$\implies$(2) in Theorem \ref{thme: main}.

\begin{theorem}\label{thm: fund-current-implies-degree-bound}
   Suppose that $X$ has a metric fundamental class. Then there exists a constant $D>0$ with the following property. If $f\in \LIP(X,\R^n)$ and $E \subset X$ are such that $f|_E$ is bi-Lipschitz, then for almost every $x\in E$ 
    $$|\iota(f,x)| \leq D.$$
    Here, $D>0$ only depends on $n$ and the constant associated to the metric fundamental class of $X$.
\end{theorem}

Let $h\colon M \to \R^n$ be Lipschitz and $x \in M$ be such that $h$ is differentiable at $x$ with non-degenerate differential. Using (3) and (4) in Lemma \ref{lemma: basic-properties-degree} one can show that if the local index is well-defined, then $\iota(h,x) = \sgn(\det(D_xh))$, cf. \cite[Theorem 5.22]{Koskela-lecture-notes}. Therefore, for closed, oriented, smooth manifolds the previous theorem holds with $D=1$ and follows by more elementary arguments that do not involve the fundamental class. Furthermore, if $D \subset M$ is a domain with $\Ha^n(\partial D)=0$, then 
$$\mu(h,D,y) = \sum_{x \in h^{-1}(y) \cap D} \iota(h,x)\quad \textrm{for almost every }y\in \R^n.$$
This and the previous observation together with the area formula yield
$$h_\#(\bb{M}\on D) = \bb{\mu(h,D,\cdot)}.$$
Next we prove this identity for metric manifolds that have a metric fundamental class.

\begin{lemma}\label{lemma: fund-current-pushforward-deg-local}
    Suppose that $X$ has a metric fundamental class $T \in \bI_n(X)$. Then, 
    $$f_\# (T\on D) = \bb{\mu(f,D, \cdot)}$$
    for every $f\in \LIP(X,\R^n)$ and every domain $D \subset X$ with $\Ha^n(\partial D) = 0$.
\end{lemma}

Note that, whenever the identity in Lemma \ref{lemma: fund-current-pushforward-deg-local} holds, we can extend it to \emph{arbitrary} open sets $U \subset M$ with $\Ha^n(\partial U)=0$ by defining $\mu(h,U,y)=\sum_i\mu(h,D_i,y)$ where $\{D_i\}$ is the countable collection of connected components of $U$. 
   
\begin{proof}
    Fix a homeomorphism $\varrho:X\to M$ of degree one and set $V=\varrho(D)$. Embed $X \subset l^\infty$. Extend $f$ to a Lipschitz map on $l^\infty$ and denote by $L$ the Lipschitz constant of the extension of $f$. Let $\epsilon>0$. Since $\varrho$ is a homeomorphism there exists $\delta>0$ such that $\varrho^{-1}(N_\delta(\partial V)) \subset N_{\epsilon/2L}(\partial D)$. By Lemma \ref{lemma: lip-approx-easy} there exist Lipschitz approximations $\varphi\colon X \to M$ and $\eta\colon M \to l^\infty$ of $\varrho$ and $\varrho^{-1}$, respectively, such that $d(\varphi,\varrho)\leq \delta$, and $\eta(y) \in N_{\epsilon/L}(\partial D)$ for every $y \in N_\delta(\partial V)$ and $g = f\circ\eta\circ\varphi$ satisfies $d(f,g) \leq \epsilon$. Therefore, $(f\circ \eta)(N_\delta (\partial V)) \subset N_\epsilon (f(\partial D))$ and the straight line homotopy $H$ between $f$ and $g$ satisfies $d(f(x),H(x,t))\leq \epsilon$ for every $x \in  D$ and each $t \in [0,1] $. It follows from Lemma \ref{lemma: currents-homotopy} that
    \begin{equation}\label{eq: fund-current-deg-local-homotopy-equal-phfwd}
        (f_\# (T\on D)) \on \R^n \setminus N_\epsilon(f(\partial D))=(g_\# (T\on D)) \on \R^n \setminus N_\epsilon(f(\partial D)).
    \end{equation}
    Moreover, by Lemma \ref{lemma: basic-properties-degree} we have
    \begin{equation}\label{eq: fund-current-deg-local-homotopy-equal-degree-1}
        \mu(f,D,y) = \mu(g,D,y)
    \end{equation} 
    for all $y \in \R^n \setminus N_\epsilon(f(\partial D))$. Let $V_\delta = \{ z \in M \colon d(z,V^c) >\delta\}$. We may assume that $\varphi$ is homotopic to $\varrho$ and thus $\deg(\varphi)=1$. Since $\varphi^{-1}(V_\delta) \subset D$, Lemma \ref{lemma: basic-properties-degree} implies $\mu(\varphi,z,D)=\deg(\varphi) = 1$ for all $z \in V_\delta$. We claim that for every $y \in \R^n \setminus N_\epsilon(f(\partial D))$ with $\#g^{-1}(y)<\infty$ we have 
    \begin{equation}\label{eq: fund-current-deg-local-homotopy-equal-degree-2}
        \mu(g,y,D) =  \mu(f\circ \eta,y,V).
    \end{equation}
        Indeed, for $y$ as above, write $(f\circ\eta)^{-1}(y)  \cap V =\{z_1,\dots,z_m\}$ and $\varphi^{-1}(z_i) = \{x_{1,i},\dots,x_{k_i,i}\}$. Then $\{z_1,\dots,z_m\} \subset V_\delta$ and $g^{-1}(y) \cap D =  \{x_{1,i},\dots,x_{k_i,i}\}$. Therefore,
    \begin{align*}
        \mu(f\circ \eta,V,y) &= \sum_{i=1}^m  \iota(f\circ\eta,z_i)
        \\
        &= \sum_{i=1}^m   \mu(\varphi,D,z_i) \cdot \iota(f\circ\eta,z_i)
        \\
        &= \sum_{i=1}^m  \sum_{j=1}^{k_i} \iota(\varphi,x_{j,i}) \cdot \iota(f\circ\eta,z_i)
        \\
        &= \sum_{i=1}^m  \sum_{j=1}^{k_i} \iota(g,x_{j,i})
        \\
        &= \mu(g,D,y).
    \end{align*}
     Combining \eqref{eq: fund-current-deg-local-homotopy-equal-degree-1} and \eqref{eq: fund-current-deg-local-homotopy-equal-degree-2} yields
    \begin{equation}\label{eq: fund-current-deg-local-homotopy-equal-degree-3}
        \mu(f,D,y) = \mu(g,D,y) = \mu(f\circ\eta,V,z)
    \end{equation}
    for almost all $y \in \R^n \setminus N_\epsilon(f(\partial D))$. Since $ \varphi_\# T = \bb{M}$, we have 
    $$\varphi_\# (T\on D) = \bb{M} - \varphi_\# (T\on (X\setminus D)).$$ 
    In particular, 
    $$(\varphi_\# (T\on D))\on V_\delta = \bb{M}\on V_\delta.$$
    
    Recall that $(f\circ \eta)(V\setminus V_\delta) \subset N_\epsilon (f(\partial D))$. Therefore,
    $$((f\circ \eta)_\# (\bb{M}\on V_\delta)) \on \R^n \setminus N_\epsilon (f(\partial D)) = \bb{\mu(f\circ\eta,V,\cdot)}\on \R^n \setminus N_\epsilon(f(\partial D)).$$
    Define $S_\epsilon = ((f\circ \eta)_\# ((\varphi_\#(T\on D)))\on V_\delta) \on \R^n \setminus N_\epsilon(f(\partial D))$. By the above and \eqref{eq: fund-current-deg-local-homotopy-equal-degree-3} we have
    $$ S_\epsilon = ((f\circ \eta)_\# ( \bb{M}\on V_\delta)) \on \R^n \setminus N_\epsilon(f(\partial D))=\bb{\mu(f,D,\cdot)}\on  \R^n \setminus N_\epsilon(f(\partial D)).$$
    Notice that
    $$(g_\#(T\on D))\on \R^n \setminus N_\epsilon(f(\partial D)) = S_\epsilon + ((f\circ \eta)_\#((\varphi_\#(T\on D))\on (M\setminus V_\delta))\on \R^n \setminus N_\epsilon(f(\partial D)).$$
    Moreover, we have that $(\varphi_\#(T\on D))\on M\setminus V_\delta$ is supported in $N_\delta(\partial V)$ and, since $(f\circ\eta)(N_\delta(\partial V))\subset N_\epsilon(f(\partial D))$, we conclude that the second term above vanishes and we have $(g_\#(T\on D))\on \R^n \setminus N_\epsilon(f(\partial D)) = S_\epsilon$.
    By \eqref{eq: fund-current-deg-local-homotopy-equal-phfwd} this shows that
    $$(f_\#(T\on D) )\on \R^n \setminus N_\epsilon(f(\partial D)) = \bb{\mu(f,D,\cdot)}\on  \R^n \setminus N_\epsilon(f(\partial D)).$$
    Since $\epsilon>0$ was arbitrary and $\Ha^n(\partial D)=0$ this proves the claim.
\end{proof}

\begin{remark}\label{rmk:multipl}
Together with the remark after the statement of Lemma \ref{lemma: fund-current-pushforward-deg-local}, the proof in fact also yields the following technical variant of Lemma \ref{lemma: fund-current-pushforward-deg-local}. 

\medskip If $S\in \bI_n(X)$ with $\partial S=0$ and $\|S\|\le C \cdot \Ha^n$ satisfies
\begin{align*}
\varphi_\#S=d\cdot \deg(\varphi)\cdot \bb M,\quad \varphi\in \LIP(X,\R^n)
\end{align*}
for some $d\in\Z\setminus\{0\}$, then $f_\#(S\on U)=d \cdot \bb{\mu(f,U,\cdot)}$ for all $f\in \LIP(X,\R^n)$ and open sets $U\subset X$ with $\Ha^n(\partial U)=0$. In light of Lemma \ref{lem:rect-criterion}, it follows from this that $d\inv \cdot S\in \bI_n(X)$. In particular $d\inv \cdot S$ is a metric fundamental class for $X$.
\end{remark}

We can now  provide the proof of Theorem \ref{thm: fund-current-implies-degree-bound}.

\begin{proof}[Proof of Theorem~\ref{thm: fund-current-implies-degree-bound}]
Since $f|_E$ is bi-Lipschitz, we have that the metric differential $\md_{f(x)}(f\inv)$ is a norm on $\R^n$ for $\Ha^n$-almost every $x\in E$. We refer to \cite{kirchheim1994rectifiable} for the theory of metric differentials. By \cite[Lemma 4]{kirchheim1994rectifiable} there exist Borel sets $E_i\subset X$ with $E=N\cup\bigcup_iE_i$, $\Ha^n(N)=0$, and norms $\|\cdot\|_i$ on $\R^n$ such that
\[
\frac 12\|x-y\|_i\le d(f\inv(x),f\inv(y))\le 2\|x-y\|_i,\quad x,y\in f(E_i)
\]
By John's ellipse theorem (see \cite[Theorem 2.18]{alvarez-thompson04}) there are inner product norms $|\cdot|_i$ on $\R^n$ so that $\frac{1}{C_n} \|\cdot\|_i\le |\cdot|_i\le C_n\|\cdot\|_i$ for all $i$, where $C_n$ is a constant depending only on $n$. If $p_i:(\R^n,|\cdot|_i)\to \R^n$ is a linear isometry, then 
\[
\frac 1{C_n} \|x-y\|\le d((p_i\circ f)\inv(x),(p_i\circ f)\inv(y))\le C_n\|x-y\|,\quad x,y\in f(E_i),
\]
and moreover $\iota(f,x)=\iota(p_i\circ f,x)$ whenever $\iota(f,x)$ is defined. Thus, replacing $E$ by $E_i$ and $f$ by $p_i\circ f$, we may suppose that $f$ is $C_n$-bi-Lipschitz on $E$, where $C_n$ is a constant only depending on $n$.

By the coarea formula the preimage $f^{-1}(x)$ is finite for almost all $x\in f(E)$. It follows that for almost all $x \in E$ the set $\{y \in X \colon f(y)=f(x)\}$ is finite. Since $T$ has a metric fundamental class there exists an integral cycle $T\in \bI_n(X)$ with $\varphi_\# T = \deg(\varphi) \cdot \bb{M}$ for every Lipschitz map $\varphi\colon X \to M$ and a constant $C>0$ such that $\norm{T} \leq C \Ha^n$. 
Let $x\in E$ and $r>0$ be such that $f^{-1}(f(x))\cap \overline{B}(x,r) = x$ and $\Ha^n(\partial B(x,r))= 0 $. Then $\mu(f,D,f(x)) = \iota(f,x)$, where $D$ is the connected component of $B(x,r)$ containing $x$. Since $\partial D\subset \partial B(x,r)$ we have $\Ha^n(\partial D) =0 $ and thus Lemma \ref{lemma: fund-current-pushforward-deg-local} and \ref{lemma: integer-currents-pushfwrd-lip-constant} yield
    $$\bb{\mu(f,D,\cdot)}(B\cap f(E))=\norm{f_\#(T\on D)}(B\cap f(E)) \leq C_n^n \norm{T}(f^{-1}(B) \cap E \cap D)$$
    for any Borel set $B\subset \R^n$. Furthermore, for almost every $y\in f(E)\setminus f(\partial D)$ we have that 
    $$|\mu(f,D,y))|=\lim_{\varepsilon\to 0}\frac{1}{\omega_n\varepsilon^n}\int_{B(y,\varepsilon)\cap f(E)}|\mu(f,D,z)|\ud\leb^n(z)$$
    by the Lebesgue differentiation theorem and the fact that almost every point of $f(E)$ is a density point. Here, \(\omega_n\) equals the Lebesgue measure of the Euclidean unit \(n\)-ball. Using that $f^{-1}:f(E)\to E$ is $C_n$-bi-Lipschitz, the estimate above gives us
    \begin{align*}
    &\lim_{\varepsilon\to 0}\varepsilon^{-n}\int_{B(y,\varepsilon)\cap f(E)}|\mu(f,D,z)|\ud\leb^n(z) \\
    \le &\lim_{\varepsilon\to 0}\varepsilon^{-n}C_n^n\norm{T}(f^{-1}(B(y,\varepsilon)\cap f(E))\cap D)\\
    \le &\lim_{\varepsilon\to 0}\varepsilon^{-n}C_n^n\|T\|(B(f^{-1}(y),C_n\varepsilon)\cap E)\\
    \le & C\cdot C_n^{2n}
    \end{align*}
    for almost every $y\in f(E)\setminus f(\partial D)$. By Lemma \ref{lemma: basic-properties-degree} the local degree $y\mapsto \mu(f,D,y)$ is constant on each component of $\R^n \setminus f(\partial D)$. Since $z:=f(x)\notin f(\partial D)$, we have that it is an interior point of $\R^n \setminus f(\partial D)$. Thus there exists $y$ as above in the same connected component of $\R^n \setminus f(\partial D)$ as $z$, giving the estimate 
    $$|\iota(f,x)| = |\mu(f,D,z) | = |\mu(f,D,y)| \leq C\cdot C_n^{2n}.$$
\end{proof}

We conclude the section with the following proposition which is a strengthening of \cite[Proposition 5.5]{basso2023geometric}. In Riemannian geometry, this statement is usually referred to as the constancy theorem.

\begin{proposition}\label{prop: constancy}
    Suppose that $X$ has a metric fundamental class $T \in \bI_n(X)$. Let $S \in \bI_n(X)$ and $U \subset X$ be a domain such that $\spt( \partial S )\cap \overline{U} = \emptyset$. Then there exists $k \in \Z$ such that $S \on U = k \cdot T\on U$.
\end{proposition}

\begin{proof}
    Fix a homeomorphism $\varrho\colon X \to M$ of degree one and put $V=\varrho (U)$. We begin with the following observation. Let $\epsilon>0$ be such that if $\varphi \colon X \to M$ is Lipschitz with $d(\varrho,\varphi)\leq \epsilon$, then $\varphi(\spt(\partial S)) \cap V= \emptyset$ and $\varphi$ is homotopic to $\varrho$. Since $\spt(\varphi_\#(\partial S)) \subset \varphi(\spt(\partial S))$, it follows from the constancy theorem in $M$ that there exists an integer $k \in \Z$ such that $$(\varphi_\# S)\on V = k \cdot \bb{M}\on V.$$  
    Furthermore, suppose that $\epsilon>0$ is sufficiently small such that the following holds. If $\psi\colon X \to M$ is another Lipschitz map satisfying $d(\varrho,\psi)\leq \epsilon$, then there exists a Lipschitz homotopy $H\colon X \times [0,1] \to M$ between $\varphi$ and $\psi$ with $H(x,t) \notin V$ for all $x \in \spt(\partial S)$ and $t \in [0,1]$. Using Proposition \ref{prop: homotopy-formula} we conclude that $k$ is independent of $\varphi$, that is,
    \begin{equation}\label{eq: constancy-thme-independence}
        (\psi_\# S) \on V = k \cdot \bb{M} \on V = (\varphi_\# S) \on V
    \end{equation}
    for any Lipschitz map $\psi\colon X \to M$ with  $d(\varrho,\psi)\leq \epsilon$. 
    
    Embed $X \subset l^\infty$. By Lemma~\ref{lemma: lip-approx-easy} there exist Lipschitz approximations $\varphi\colon X\to M$ and $\eta\colon M\to l^\infty$ of $\varrho$ and $\varrho^{-1}$, respectively, such that the Lipschitz map $f= \eta\circ \varphi$ satisfies $d(f(x), x)\leq \epsilon$ for every $x\in X$ and the observation above applies to $\varphi$. Furthermore, we can arrange it such that for $\delta>0$ sufficiently small we have
    \begin{equation}\label{eq:const-thm-close-supp}
        U_\epsilon \subset \varphi^{-1}(V_\delta) \subset N_\epsilon(U),
    \end{equation}
    where $V_\delta = \{ y\in M \colon d(y,V^c)> \delta\}$ and analogously for $U_\epsilon$. Set $Q = S\on U - k \cdot T \on U$. We may assume that $Q \in \bI_n(X)$ and $\spt \partial Q \subset \partial U$. Otherwise, we replace $U$ by $N_\sigma(U)$ for $\sigma>0$ sufficiently small. We first show that $\partial Q = 0$. The slicing theorem \cite[Theorem 5.6 and 5.7]{ambrosio-kirchheim-2000} implies that there exists $\delta>0$ such that (\ref{eq:const-thm-close-supp}) holds and 
    $$W = S\on \varphi^{-1}(V_\delta) - k \cdot T \on \varphi^{-1}(V_\delta)  \in \bI_n(X).$$
    By \eqref{eq: constancy-thme-independence} we have $\varphi_\#W  = 0$ and thus $f_\#W =0$. Furthermore, since $l^\infty$ is injective, there exists a Lipschitz homotopy $H\colon [0,1]\times X\to l^\infty$ between $f$ and identity on $X$ satisfying $d(H(x,t),x)\leq \epsilon$ for every $x \in X$ and every $t\in [0,1]$. It follows from Proposition \ref{prop: homotopy-formula} that the integral $n$-current $R=H_\#(\bb{0,1}\times \partial W)$ is a filling of $\partial W$, that is, $\partial R = \partial W$ and $R$ supported in $N_\epsilon(\spt \partial W)$. In particular, the integral $n$-current $Q-W-R$ is a filling of $\partial Q$ with support in the closed $2\epsilon$-neighborhood of $\spt(\partial Q)$. Since $\epsilon>0$ was arbitrary, \cite[Lemma 5.6] {basso2023geometric} implies that $\partial Q=0$. If $U=X$, then ones shows exactly as in \cite[Proposition 5.5]{basso2023geometric} that $Q=0$. Otherwise, if $U \neq M$, then $\spt (Q)\neq X$. Therefore, for $\epsilon>0$ sufficiently small we have $\spt(\varphi_\# Q) \neq M$ whenever $\varphi\colon X \to M$ is a Lipschitz map satisfying $d(\varrho,\varphi) \leq \epsilon$. In particular, $\varphi_\# Q = 0$ for such $\varphi$. Analogously as above, using Proposition \ref{prop: homotopy-formula}, we obtain fillings of $Q$ in $N_\epsilon(U)$ for $\epsilon>0$ arbitrarily small. It follows again from \cite[Lemma 5.6]{basso2023geometric} that $Q=0$. This implies $S \on U = k \cdot T\on U$, as claimed.    
\end{proof}

\section{Perturbations and projections}\label{sec:projection-lemma}
In this section we collect a few tools that will be used in the sequel. The following perturbation result is an adaptation of Theorem \ref{thme:weak-perturb-bate}.

\begin{proposition}\label{prop: weak-perturb-simplicial}
    Let $X$ be a compact metric space with finite Hausdorff $n$-measure and $\Sigma$ be a finite 
    simplicial complex equipped with the $l^2$-metric $|\cdot|_{l^2}$. Furthermore, let $\varphi \colon X \to \Sigma$ be a $C$-Lipschitz map and $P \subset X$ be purely $n$-unrectifiable. Then, for any $\epsilon>0$ there exists a $3C$-Lipschitz map $\psi\colon X \to \Sigma$ satisfying
    \begin{enumerate}
        \item $d(\varphi,\psi) \leq \epsilon$;
        \item $\psi\left(\varphi^{-1}(\Delta)\right) \subset \Delta$ for each simplex $\Delta$ in $\Sigma$;
        \item $\Ha^n(\psi(P)) \leq \epsilon$.
    \end{enumerate}
\end{proposition}
\begin{proof}
    Let $\epsilon>0$. Denote by $\mathcal{S}$ the collection of all simplices in $\Sigma$ of dimension at least $n$ and that are not contained in a simplex of higher dimension. For $\Delta \in \mathcal{S}$ and $\delta>0$, let $K_\Delta = \varphi^{-1}(\Delta)$ and $K_{\Delta,\delta} = \varphi^{-1}(\Delta_\delta)$, where $\Delta_\delta = \big\{y \in \Delta \colon d(y,\partial \Delta)\geq \delta \big\}$. Now, let $\delta>0$ be sufficiently small such that $\Ha^n\left(K_\Delta\setminus K_{\Delta,\delta}\right) \leq \epsilon$  for every $\Delta\in \mathcal{S}$. It follows from Theorem \ref{thme:weak-perturb-bate} that for each $\Delta \in \mathcal{S}$ there exists a $2C$-Lipschitz map $\psi_\Delta\colon K_\Delta \to \Delta$ with the following properties:
    \begin{enumerate}[label=(\roman*)]
        \item $|\varphi(x)-\psi_\Delta(x)|_{l^2}\leq \sigma$ for each $x \in K_\Delta$ and $\varphi(x)=\psi_\Delta(x)$ \\whenever $d(x,P\cap K_{\Delta,\delta})> \sigma $;
        \item $\Ha^n(\psi_\Delta(P \cap K_{\Delta,\delta})) \leq \sigma$.
    \end{enumerate}
    Here, $\sigma>0$ is such that $2C\sigma \leq \min\left\{\epsilon, \delta\right\}$. Let $\Delta \in \mathcal{S}$. Then,
    \begin{equation}\label{eq: weak-pert-simplicial-lower-bound}
        \frac{\delta}{2} \leq |\varphi(x)-\varphi(y)|_{l^2}\leq C d(x,y)
    \end{equation}
    for every $x \in \varphi^{-1}\left(\Sigma\setminus \Delta_{\delta/2}\right)$ and every $y\in K_{\Delta,\delta}$. Therefore, $\varphi(x) = \psi_\Delta(x)$ for every $x \in \varphi^{-1}(\Delta\setminus \Delta_{\delta/2})$. This implies $\psi_\Delta(\varphi^{-1}(\Delta)) \subset \Delta$ for each simplex $\Delta$ in $\Sigma$. Define $\psi \colon X \to \Sigma$ by $\psi(x) = \psi_\Delta(x)$ if $x \in K_\Delta$. By the above $\psi$ is $2C$-Lipschitz on each $K_\Delta$ and satisfies (1) as well as (2). We claim that $\psi$ is $3C$-Lipschitz. Indeed, let $\Delta, \Delta' \in \mathcal{S}$ be distinct. If $x \in K_\Delta\setminus K_{\Delta,\delta/2}$ and $y \in K_{\Delta'}\setminus  K_{\Delta',\delta/2}$, then
    $$|\psi(x)-\psi(y)|_{l^2} = |\varphi(x)-\varphi(y)|_{l^2} \leq C d(x,y).$$
    On the other hand, if $x \in K_{\Delta,\delta/2}$ or $y \in K_{\Delta',\delta/2}$, then one shows exactly as in $\eqref{eq: weak-pert-simplicial-lower-bound}$ that $\delta\leq 2Cd(x,y)$. It follows from (i) that 
    $$        |\psi(x)-\psi(y)|_{l^2} \leq |\psi(x)-\varphi(x)|_{l^2} + |\varphi(x)-\varphi(y)|_{l^2} + |\varphi(y)-\psi(y)|_{l^2} \leq 3C d(x,y).$$
    Finally, by the choice of $\delta$ and (ii) we have
    $$ \Ha^n\left(\psi(P)\right) \leq  \sum_{\Delta\in \mathcal{S}} \Ha^n\left(\psi(P \cap K_{\Delta,\delta})\right) + \Ha^n\left(\psi((K\setminus K_{\Delta,\delta}))\right) \leq 3|\mathcal{S}| C \epsilon.$$
    Since $\epsilon>0$ was arbitrary, this shows that $\psi$ satisfies (3) and completes the proof.
\end{proof}

We will also need the following version of the deformation theorem.
\begin{lemma}\label{lemma: deformation-projection}
    Let $\Sigma$ be a finite simplicial complex of dimension $N$ and equipped with the $l^2$-metric. Suppose that there exist finite Borel measures $\mu_1, \mu_2$ on $\Sigma$ that are supported in a closed set $K \subset \Sigma$ with finite Hausdorff $n$-measure and $n < N$. Then, there exists a map $p \colon \Sigma \to \Sigma$ that restricts to the identity on the $(N-1)$-skeleton of $\Sigma$ and $p(\Delta)\subset \Delta$ for each simplex $\Delta$ in $\Sigma$. Furthermore, $p|_{K}$ is Lipschitz, and $p(K) \subset \Sigma^{N-1}$ and
    \begin{equation}\label{eq: deformation-projection}
        \int_{\inte\Delta} (\Lip(p)(x))^n \; d\mu_i(x) \leq C \mu_i(\inte\Delta)
    \end{equation}
    for $i=1,2$ and every simplex $\Delta$ in $\Sigma$. Here, $C>0$ depends only on $N$.
\end{lemma}

\begin{proof}
    Let $\Delta\in\Sigma^N$. By the same argument as in \cite{fed-flem} there exists a Borel subset $A \subset \inte\Delta$ with positive and finite Hausdorff $N$-measure and the following property. For every $x \in A$, the radial projection $p$ with $x$ as center satisfies
    \begin{equation}\label{eq: easy-deform}
        \int_{\inte\Delta} (\Lip(p)(y))^{n} \; d\mu_i(y) \leq C \mu_i({\inte\Delta}),
    \end{equation}
    for $i=1,2$ and a constant $C>0$ only depending on $N$. Since $\Ha^N(K)= 0$ and $X$ is compact, it follows that for each $\Delta\in\Sigma^N$ there exists a radial projection $p_\Delta$ satisfying \eqref{eq: easy-deform} and such that the projection center $x_\Delta$ of $p_\Delta$ satisfies $x_\Delta \in {\inte\Delta} \setminus K$. In particular, $p_\Delta|_{K}$ is Lipschitz for each $\Delta\in\Sigma^N$. Furthermore, each projection $p_\Delta$ restricts to the identity on the $(N-1)$-skeleton of $\Sigma$ and $p_\Delta(\Delta)\subset \Delta$ for each simplex $\Delta$ in $\Sigma$. Therefore, we obtain a map $p$, as claimed, by defining $p=p_\Delta$ on each $\Delta \in \Sigma^N$ and defining $p$ as the identity elsewhere.
\end{proof}

The maps given by the previous lemma are the building blocks to prove the next proposition.

\begin{proposition}\label{prop: deformation-projection-no-rect}
    Let $X$ be a compact metric space with finite Hausdorff $n$-measure and $\Sigma$ be a finite simplicial complex of dimension $N \geq n$ and equipped with the $l^2$-metric. Let $P \subset X$ be purely $n$-unrectifiable. If there exists a $L$-Lipschitz map $\varphi\colon X \to \Sigma$ such that $\Ha^n(\varphi(P)) \leq \epsilon$, then there exists a map $p \colon \Sigma \to \Sigma$ with $p(\Delta) \subset \Delta$ for each simplex $\Delta$ in $\Sigma$ and $p \circ \varphi$ is Lipschitz with image in the $n$-skeleton of $\Sigma$. Moreover, $\Ha^n(p(\varphi(P))) \leq C\epsilon$ and for each $\Delta \in \Sigma^n$ we have
    \begin{equation}\label{eq:deformation-no-rect}
        \int_{\Delta} N(p \circ \varphi,y) \; d\Ha^n(y) \leq C \Ha^n(\varphi^{-1}(\st\Delta)).
    \end{equation}
    Here, the constant $C>1$ only depends on $L,N$ and $n$.
\end{proposition}

\begin{proof}
    The proof is by induction over the dimension of the skeleton of $\Sigma$. First, we construct a map $p\colon \Sigma \to \Sigma$ that has the desired properties for the $(N-1)$-skeleton $\Sigma^{N-1}$ instead of the $n$-skeleton $\Sigma^n$. Set $\mu_1 = N(\varphi, \cdot ) \Ha^n$ and $\mu_2 = \Ha^n\on  \varphi(P)$. Since $\varphi$ is Lipschitz, it follows from the coarea formula that $\mu_1$ and $\mu_2$ are finite Borel measures on $\Sigma$ supported in the compact set $\varphi(X)$ with finite Hausdorff $n$-measure. Let $p$ and $C(N)$ be the map and constant obtained by Lemma \ref{lemma: deformation-projection}, respectively. Clearly, the first three properties are satisfied by $p$. Let $\Delta \in \Sigma^N$. The area formula and \eqref{eq: deformation-projection} imply      
    \begin{align*}
        \int_{\partial\Delta} N(p\circ \varphi,x,\Delta) \; d\Ha^n(x) &= \int_{\partial\Delta} \sum_{y \in p^{-1}(x)\cap \Delta} N(\varphi,y) \; d\Ha^n(x) 
        \\
        &= \int_\Delta \mathbf{J}(D_y p) N(\varphi,y) \; d\Ha^n(y)
        \\
        &\leq \int_{\inte \Delta } (\Lip(p)(y))^{n} \; d\mu_1(y) + \int_{\partial \Delta}  N(\varphi,y) \; d\Ha^n(y)
        \\
        &\leq C(N) \mu_1(\Delta) = C(N)\int_\Delta N(\varphi,y) \; d\Ha^n(y).
    \end{align*}
    
    Therefore, for each $\Delta\in \Sigma^{N-1}$ we have 

    \begin{equation}\label{eq:deform-no-rect-induction}
        \int_{\Delta} N(p \circ \varphi,x) \; d\Ha^n(x) \leq C(N) \int_{\st\Delta} N(\varphi,y) \; d\Ha^n(y).
    \end{equation}

    Moreover, by an analogous argument as above,
    \begin{align*}
        \Ha^n(p(\varphi(P))) \leq \int_{\Sigma^{N-1}} N(p,x,\varphi(P)) \; d\Ha^n(x) \leq C(N) \Ha^n(\varphi(P)).
    \end{align*}

    For the induction step, suppose that for some value $k< N-n$ there exists a map $p\colon \Sigma \to \Sigma$ satisfying the desired properties for the $(N-k)$-skeleton and some constant $C>0$ but with \eqref{eq:deform-no-rect-induction} instead of \eqref{eq:deformation-no-rect}. Define $\mu_1 = N(p\circ \varphi,\cdot) \Ha^n$ and $\mu_2 = \Ha^n\on p(\varphi(P))$. By the induction hypothesis, we have that $\mu_1$ and $\mu_2$ are finite Borel measures supported in a compact set of finite Hausdorff $n$-measure which is contained in the $(N-k)$-skeleton. Let $\pi\colon \Sigma^{N-k} \to \Sigma^{N-k}$ and $C(k)$ be the map and constant obtained by Lemma \ref{lemma: deformation-projection} applied to $\mu_1,\mu_2$ and $\Sigma^{N-k}$. An analogous argument as above shows that $\Ha^n(\pi(p(\varphi(P)))) \leq C(k) C \Ha^n(\varphi(P))$. Furthermore, by the area formula, and \eqref{eq: deformation-projection} and the induction hypothesis we conclude that for each simplex $\Delta \in \Sigma^{N-k-1}$
    \begin{align*}
        \int_{\partial\Delta} N(\pi \circ (p\circ \varphi),x, \Delta) \; d\Ha^n(x) &= \int_{\partial\Delta} \sum_{y \in \pi^{-1}(x) \cap \Delta} N(p\circ \varphi,y) \; d\Ha^n(x) 
        \\
        &\leq C(k) \int_\Delta N(p\circ\varphi,y) \; d\Ha^n(y)
        \\
        &\leq C(k)C \int_{\st\Delta} N(\varphi,y) \; d\Ha^n(y).
    \end{align*}
    From this \eqref{eq:deformation-no-rect} is easily verified using the coarea inequality.
\end{proof}

\section{Approximation by Lipschitz manifolds}\label{sec: approx-lip-mfld}
Throughout this section, let $X$ be a metric $n$-manifold such that $\Ha^n(X)<\infty$ and $\dim_NX<\infty$. Suppose that $\varrho \colon X \to M$ is a homeomorphism onto the closed, oriented, smooth $n$-manifold $M$.

\medskip The goal of this section is to prove the following theorem, which proves the implication (2)$\implies$(3) in Theorem \ref{thme: main}.  In the statement, we assume that $X$ satisfies Theorem \ref{thme: main}(2), that is, there exists $D\ge 1$ so that $|\iota(f,x)|\le D$ for almost all $x\in E$ whenever $f\in \LIP(X,\R^n)$ and $E\subset X$ are such that $f|_E$ is bi-Lipschitz.

\begin{theorem}\label{thm: lip-mfld-approx-only-degree}
    There exists a sequence of metric spaces $X_k$ and bi-Lipschitz maps $\psi_k\colon M \to X_k$ such that the following properties hold.
    \begin{enumerate}
        \item The homeomorphism $\varphi_k= \psi_k \circ \varrho$ are $\epsilon_k$-isometries $\varphi_k \colon X \to X_k$ with $\epsilon_k \to 0$ as $k \to 0$;
        \item There is a constant $C>1$ such that
        $$\limsup_{k \to \infty} \Ha^n(\varphi_k(K)) \leq C \Ha^n(K)$$
         for every compact $K\subset X$.
    \end{enumerate}
\end{theorem}
Here, the constant $C$ depends only on $D$ and the data of $X$ (i.e. $n,\ \dim_NX$ and the constants appearing in the definition of Nagata dimension).

\medskip For the proof of Theorem \ref{thm: lip-mfld-approx-only-degree}, fix $\epsilon>0$ and write $N= \dim_N X$ and let $X = E \cup P $ be a partition, where $E$ is $n$-rectifiable and $P$ is purely $n$-unrectifiable. We embed $X \subset l^\infty$. First, we prove the following lemma, which introduces the objects we need to construct the manifolds and bi-Lipschitz homeomorphism in Theorem \ref{thm: lip-mfld-approx-only-degree}.

\begin{lemma}\label{lemma: lip-mfld-approx-setup}
    There exists a finite simplicial complex $\Sigma$ of dimension $\leq N$ equipped with the $l^2$-metric $|\cdot|_{l^2}$ and every simplex in $\Sigma$ has side length $\epsilon$. Furthermore, there exist a constant $C>0$, only depending on the data of $X$, and Lipschitz maps $f,h\colon X \to \Sigma $ and $g\colon \Sigma \to l^\infty$ with the following properties:
    \begin{enumerate}
    \item $f$ is $C-$Lipschitz and $g$ is $C$-Lipschitz on each simplex $\Delta$ in $\Sigma$;
    \item $d(x,g(f(x))) \leq C\epsilon$ and $d(x,g(h(x))) \leq C\epsilon$ for every $x \in X$;
    \item the image of $h$ is contained in the $n$-skeleton $\Sigma^n$ of $\Sigma$;
    \item $f(x)$ and $h(x)$ are contained in the same simplex for each $x \in X$;
    \item $\Ha^n(h(P)) < \epsilon^n/(2n!)$;
    \item for each $\Delta \in \Sigma^n$
    $$ \int_{\Delta} N(h,y) \; d\Ha^n(y) \leq C \Ha^n(f^{-1}(\st\Delta)).$$
\end{enumerate}

\end{lemma}

\begin{proof}
It follows from Theorem \ref{thm:simplicial-factorization}, that there exists a finite simplicial complex $\Sigma$ of dimension $\leq N$ equipped with the $l^2$-metric and every simplex in $\Sigma$ has side length $\epsilon$. Furthermore, there is a constant $C_1>0$, only depending on the data of $X$, and Lipschitz maps $f\colon X \to \Sigma $ and $g\colon \Sigma \to l^\infty$ satisfying $d(x,g(f(x))) \leq C_1\epsilon$ for all $x \in X$, and $f$ is $C_1$-Lipschitz and $g$ is $C_1$-Lipschitz on every simplex $\Delta$ in $\Sigma$. By first using Proposition \ref{prop: weak-perturb-simplicial} to obtain a suitable approximation of $f$ and then applying Proposition \ref{prop: deformation-projection-no-rect} to this approximation, we conclude that there exists a Lipschitz map $h \colon X \to \Sigma$ satisfying (3)-(5) and (6) with a constant $C_2>0$. Notice that $C_2>0$ depends only on $C_1,N$ and $n$. We claim that $h$ also satisfies (2). Indeed, let $x \in X$. Since $g$ is $C_1$-Lipschitz on every simplex and satisfies (2), we have
$$d(x,g(h(x))) \leq d(x,g(f(x))) + d(g(f(x)),g(h(x)))\leq C_1 \epsilon + C_1 \diam \Delta = C_3 \epsilon.$$
The statement follows for $C = \max\{C_1,C_2,C_3\}$.
\end{proof}

Next, we construct a Lipschitz approximation of $\eta = h \circ \varrho^{-1}$ that also satisfies (6). We use a similar construction as in \cite[Section 5]{Meier-Wenger}, see also \cite{Brian-white-least-volume}. We remark that this is the only place where we need the local index bounds.
\begin{lemma}\label{lemma: good-lip-approx-into-simplicial-complex}
    There exists a Lipschitz map $\psi\colon M \to \Sigma$ with image contained in the $n$-skeleton $\Sigma^n$ of $\Sigma$ and such that $\psi(\varrho^{-1}(x))$ and $h(x)$ are contained in simplices with a common side for each $x\in X$. Moreover,
   \begin{equation}\label{eq: good-lip-approx-into-simplicial-complex}
       \int_\Delta N(\psi,y) \; d\Ha^n(y) \leq 2CD\cdot \Ha^n(f^{-1}(\st\Delta))
   \end{equation}
   for every $\Delta \in  \Sigma^n$.
\end{lemma}

\begin{proof}
    Let $\Delta \in \Sigma^n$. It follows from (3) and (5) of Lemma \ref{lemma: lip-mfld-approx-setup} that there exists a Borel set $A \subset \Delta$ with $2\Ha^n(A) > \Ha^n(\Delta) =\epsilon^n/n!$ and such that for each $z \in A$ we have $h^{-1}(z) \subset E$. Since $X$ satisfies Theorem \ref{thme: main}(2) and $h$ is Lipschitz, Lemma \ref{lemma: equiv-def-degree-bound} implies
    $$\Ha^n\left(\{ z \in A \colon \#h^{-1}(z) < \infty \textup{ and } |\iota(h,x)|\le D \textup{ for each } x\in h\inv(z) \}\right) > \frac{1}{2} \Ha^n(\Delta).$$
    Moreover, 
    $$\Ha^n\left(\left\{ z \in \Delta \colon  \Ha^n(\Delta) \cdot N(h,z) \leq 2\int_{\Delta} N(h, z) \; d\Ha^n(z) \right\}\right) \geq \frac{1}{2}\Ha^n(\Delta).$$
    Therefore, by (6) of Lemma \ref{lemma: lip-mfld-approx-setup} we conclude that for every simplex $\Delta_i \in \Sigma^n$ there exists a point $z_i \in \inte \Delta$ satisfying
    \begin{equation}\label{eq: good-lip-approx-into-simplicial-complex-bound-mult-points}
        \Ha^n(\Delta_i) \cdot  N(h,z_i) \leq 2\int_{\Delta_i} N(h, z) \; d\Ha^n(z) \leq 2C \cdot \Ha^n(f^{-1}(\st(\Delta_i)))
    \end{equation}
    and for each $y \in h^{-1}(z_i)$ the local index of $h$ at $y$ is well defined with  $|\iota(h,y)|\leq D$. Let $\eta = h \circ \varrho^{-1}$. We have $\iota(h,y) = \iota(\eta,\varrho(y))$ and $N(h,z_i)=N(\eta,z_i)$ for all $z_i$ and every $y \in h^{-1}(z_i)$ because $\varrho\colon X \to M$ is a homeomorphism. If $N(\eta,z_i) = 0$, then $z_i\notin \im \eta$ and we define $\psi$ on $\eta^{-1}(\Delta_i)$ as the composition of $\eta$ with the radial projection with $z_i$ as center. Otherwise, write $\eta^{-1}(z_i) = \{x_{i,1},\dots,x_{i,m_i}\}$, where $m_i = N(\eta,z_i)$. For $x_{i,j}$, let $B_{i,j}$ be an open set in $\eta^{-1}(\textup{int}(\Delta_i))$ that is the bi-Lipschitz homeomorphic image of the Euclidean ball $B^n(0,1)$ and such that the family $\left\{B_{i,j}\right\}_{i,j}$ is pairwise disjoint. We moreover denote by $\frac{1}{2}B_{i,j}$ the image of $B^n(0,1/2)$ under the same bi-Lipschitz homeomorphism. We define $\psi$ outside $B_{i,j}$ as the composition of $\eta$ and the radial projection with the $z_i$ as center points. Denote $D_{i,j}:=\iota(\eta,x_{i,j})$. If $D_{i,j}= 0$, there exists a continuous extension of $\psi$ to $B_{i,j}$ with image inside $\partial \Delta_i$. If   $D_{i,j}\neq 0$, let $\psi_{B_{i,j}\setminus \frac{1}{2}B_{i,j}}$ be a continuous homotopy inside $\partial \Delta_i$ between $\psi|_{\partial B_{i,j}}$ and a standard Hopf map $H_{i,j}:\partial \frac 12 B_{i,j}\to \partial \Delta_i$ of degree $D_{i,j}$ (which is Lipschitz). Finally, let $\psi_{\frac{1}{2}B_{i,j}}:\frac{1}{2}B_{i,j}\to \Delta_i$ be a Lipschitz extension of $H_{i,j}$ for which every point has exactly $|D_{i,j}|$ preimages. It follows that the map $\psi\colon M \to \Sigma$ is continuous with image in the $n$-skeleton of $\Sigma$. Furthermore, the restriction $\psi_{\frac{1}{2}B_{i,j}}$ is Lipschitz for all $(i,j)$ and $\psi(\Omega) \subset \Sigma^{n-1}$, where $\Omega = M \setminus \bigcup_{i,j} \frac 12B_{i,j}$. The $(n-1)$-skeleton of $\Sigma$ is locally Lipschitz $k$-connected for all $k\geq 0$ and $\psi$ restricted to $\partial \Omega$ is Lipschitz. Therefore, we can replace $\psi$ on $\Omega$ by a Lipschitz map approximation arbitrarily close to the original map and which agrees with $\psi$ on $\partial \Omega$. Thus we can guarantee that $\psi(\varrho^{-1}(x))$ and $h(x)$ are contained in simplices with a common side for each $x\in X$. Since $|D_{i,j}|\leq D$ for every $(i,j)$ and by the definition of $\psi$ we have
    $$\int_{\Delta_i} N(\psi,y) \; d\Ha^n = \int_{\Delta_i} \sum_{j=1}^{m_i} |D_{i,j}| \; d\Ha^n \leq  D \cdot N(\eta,z_i) \cdot \Ha^n(\Delta_i).$$
    for every $\Delta_i$. Therefore, \eqref{eq: good-lip-approx-into-simplicial-complex-bound-mult-points} implies 
    $$\int_{\Delta_i} N(\psi,y) \; d\Ha^n \leq 2CD \cdot \Ha^n(f^{-1}(\st(\Delta_i)))$$ 
    for every $\Delta_i$. This completes the proof.
\end{proof}

We are now in a position to construct the Lipschitz manifold $Y$ as in the statement of Theorem \ref{thm: lip-mfld-approx-only-degree}. Let $\psi\colon M \to \Sigma$ be the Lipschitz map obtained by the previous lemma. We define
    \begin{align*}
        \Psi\colon &M \to M\times l^\infty,
        \\
        &x \mapsto (x,g(\psi(x))).
    \end{align*}
    Set $Y = \Psi(M)$ and equip $Y$ with the metric defined as follows $$d_Y((x,v),(y,w)) = \epsilon d_M(x,y) + d_{l^\infty}(v,w).$$ Since $\psi$ and $g$ are Lipschitz, the map $\Psi\colon M \to Y$ is a bi-Lipschitz homeomorphism.  Denote by $\varphi := \Psi \circ \varrho^{-1}$ the homeomorphism $M\to Y$ obtained as the composition of $\Psi$ and $\varrho\inv$.

\begin{lemma}\label{lemma: lip-mfld-approx-rough-isom}
    There exists $c_1>0$ only depending on $C>0$ and $n$ such that 
    $$|d(x,y) - d(\varphi(x),\varphi(y))|\leq (\diam M +  c_1)\epsilon$$
    for every $x,y \in X$.
\end{lemma}

\begin{proof}
    Let $x \in X$. By (3) of Lemma \ref{lemma: lip-mfld-approx-setup} there exist $\Delta_1 \in \Sigma^n$ and $\Delta_2 \in  \mathcal{F}(\Sigma)$ with $\Delta_1\cap \Delta_2 \neq \emptyset$ and such that $h(x) \in \Delta_1$ and $(\psi \circ \varrho^{-1})(x) \in \Delta_2$, respectively. It follows from Lemma \ref{lemma: close-points-neighb-simplicies} that there exists $z \in \Delta_1\cap \Delta_2$ satisfying
    $$|(\psi \circ \varrho^{-1})(x)-z|_{l^2}+|z-h(x)|_{l^2} \leq 4 \sqrt{n} |(\psi \circ \varrho^{-1})(x)-h(x)|_{l^2}.$$
    Recall that each simplex in $\Sigma$ has side-length $\epsilon$. Therefore, using that $g$ is $C$-Lipschitz on each simplex, we get
    $$d(g(h(x)),g(\psi(\varrho^{-1}(x)))) \leq 8 \sqrt{n} C\epsilon.$$
    In particular, for each $x,y \in X$ we have    
    \begin{align*}
        &|d_X(x,y)-d_Y(\varphi(x),\varphi(y))|
        \\
        \leq &\epsilon \diam M+ |d_X(x,y)-d_{l^\infty}(g(\psi(\varrho^{-1}(x))),g(\psi(\varrho^{-1}(y)))|
        \\
        \leq &\epsilon (16 \sqrt{n}C + \diam(M)) + |d_X(x,y)-d_{l^\infty}(g(h(x)),g(h(y)))|
        \\
        \leq &(\diam(M)+18 \sqrt{n}C)\epsilon = (\diam M +  c_1)\epsilon.
    \end{align*}
\end{proof}

We now turn our attention to the second property.

\begin{lemma}\label{lemma: lip-mfld-approx-volume-bound}
    There exists $c_2>0$ only depending on $C,D, N$ and $n$ such that
    $$\Ha^n(\varphi(K)) \leq \int_{\varrho^{-1}(K)} \mathbf{J}(D_z \Psi) \; d\Ha^n(z) \leq \epsilon^n \Ha^n(M) + c_2 \Ha^n(N_{c_2\epsilon }(K))$$
    for every compact $K \subset X$.
\end{lemma}

\begin{proof}
    Let $K \subset X$ be compact. It follows from the area formula that
    \begin{align*}
        \Ha^n(\varphi(K)) &= \Ha^n(\Psi(\varrho^{-1}(K)))
        \\
        &= \int_Y N(\Psi, y, \varrho^{-1}(K)) \; d\Ha^n(y) 
        \\
        &= \int_{\varrho^{-1}(K)} \textbf{J}(D_x \Psi) \; d\Ha^n(x).
    \end{align*}
    It is not difficult to show that $\textbf{J}(D_x \Psi) \leq \epsilon^n + C^n\textbf{J}(D_x \psi)$ for almost all $x \in M$. Furthermore, the area formula and \eqref{eq: good-lip-approx-into-simplicial-complex} yield
    \begin{align*}
        \int_{\varrho^{-1}(K)} \textbf{J}(D_x \psi) \; d\Ha^n(x) &=\int_\Sigma N(\psi, y, \varrho^{-1}(K))\; d\Ha^n(y)
         \\
         &= \sum_{\Delta \in \mathcal{S}} \int_\Delta N(\psi, y, \varrho^{-1}(K))\; d\Ha^n(y)
         \\
         &\leq  2CD \sum_{\Delta \in \mathcal{S}}  \Ha^n(f^{-1}(\st\Delta)).
    \end{align*}
    Here, $\mathcal{S}$ denotes the set of all $n$-dimensional simplices intersecting $\psi(\varrho^{-1}(K))$. We claim that
    $$\sum_{\Delta \in \mathcal{S}}  \Ha^n(f^{-1}(\st\Delta)) \leq c \cdot \Ha^n(N_{c \epsilon}(K))$$
    where $ c$ only depends on $C, N$ and $n$. Indeed, for a simplex $\sigma$ in $\Sigma$ the number of distinct $n$-dimensional simplices $\Delta \in \Sigma^n$ satisfying $\sigma \subset \st \Delta$ is bounded by a constant $c(N,n)$ that depends only on $N$ and $n$. Therefore,
    $$\sum_{\Delta \in \mathcal{S}}  \Ha^n(f^{-1}(\st\Delta)) \leq c(N,n) \cdot\Ha^n\left(f^{-1}\left(\bigcup_{\Delta \in \mathcal{S}}\st\Delta\right)\right).$$
    Let $x \in f^{-1}\left(\bigcup_{\Delta \in \mathcal{S}}\st\Delta\right)$. It follows from the definition of $\mathcal{S}$ that there exist $\Delta\in \mathcal{S}$ and $y \in K$ such that $f(x), \psi(\varrho^{-1}(y)) \in \Delta$. Recall that $f(y)$ and $\psi(\varrho^{-1}(y))$ are contained in simplices with a common side. Thus, there exists a simplex $\Delta' \subset \Sigma$ with $\Delta \cap \Delta' \neq \emptyset$ and $f(y) \in \Delta'$. By Lemma \ref{lemma: close-points-neighb-simplicies} there exists $z \in \Delta \cap \Delta'$ satisfying 
    $$|(\psi \circ \varrho^{-1})(y)-z|_{l^2}+|z-f(y)|_{l^2} \leq 4 \sqrt{N} |(\psi \circ \varrho^{-1})(y)-f(y)|_{l^2}.$$
    Since $g$ is $C$-Lipschitz on each simplex and $d(x,g(f(x))) \leq C \epsilon$ for every $x \in X$, we conclude
    \begin{align*}
        d(x,y) & \leq d(x,g(f(x))) + d(g(f(x)),g(f(y))) + d(g(f(y)),y)
        \\
        &\leq 2C\epsilon + d(g(f(x)),g(\psi(\varrho^{-1}(y)))) + d(g(\psi(\varrho^{-1}(y))),g(f(y)))
        \\
        &\leq 10 C \sqrt{N}\epsilon.
    \end{align*}
    In particular, $f^{-1}\left(\bigcup_{\Delta \in \mathcal{S}}\st\Delta\right) \subset N_{10 C \sqrt{N}\epsilon}(K)$. Putting everything together
    $$\Ha^n(\varphi(K)) \leq \int_{\varrho^{-1}(K)} \mathbf{J}(D_z \Psi) \; d\Ha^n(z) \leq \epsilon^n \Ha^n(M) + c_2 \Ha^n(N_{c_2\epsilon }(K)),$$
    where $c_2 = \max\{2c(N,n)C^{n+1}D,10 C \sqrt{N}\}$.
\end{proof}

\begin{proof}[Proof of Theorem \ref{thm: lip-mfld-approx-only-degree}]
Combining Lemma \ref{lemma: lip-mfld-approx-rough-isom} and \ref{lemma: lip-mfld-approx-volume-bound}
we conclude that there exist a metric space $Y$ that is bi-Lipschitz to $M$ and a homeomorphism $\varphi\colon X \to Y$ satisfying 
$$|d(x,y) - d(\varphi(x),\varphi(y))| \leq (\diam M +  C)\epsilon$$
for all $x,y \in X$ and  for every compact $K \subset X$
$$\Ha^n(\varphi(K)) \leq \epsilon^n \Ha^n(M) + C \Ha^n(N_{C\epsilon }(K)),$$ 
where the constant $C$ depends only on $D$ and the data of $X$. Since $\epsilon>0$ was arbitrary this completes the proof of Theorem \ref{thm: lip-mfld-approx-only-degree}.
\end{proof}

\section{Construction of a metric fundamental class}\label{sec:approx-lip-mfld-to-fund-class}
Finally, we prove that if there exists a sequence of metric spaces $X_k$ approximating $X$ in the Gromov-Hausdorff distance with volume control, then $X$ has a metric fundamental class. Together with Theorem \ref{thm: fund-current-implies-degree-bound} and \ref{thm: lip-mfld-approx-only-degree} this completes the proof of Theorem \ref{thme: main}. The following theorem yields the implication (3)$\implies$(1), however its statement is more general. In particular, the approximating manifolds need not belong to a single homeomorphism class.

\begin{theorem}\label{thm: lip-mfld-approx-implies-fund-current}
Let $X_k$ be closed, oriented, metric $n$-manifolds admitting a metric fundamental class with the constant in Definition \ref{def: metric fundamental class}(a) independent of $k$. Suppose that $\psi_k:X_k\to X$ are continuous $\epsilon_k$-isometries, $\epsilon_k\to 0$, with $\deg(\psi_k)\ne 0$,  where $X$ is a closed, oriented, metric $n$-manifold with $\Ha^n(X)<\infty$, such that
\begin{align}\label{eq:vol-control}
\limsup_{k\to\infty}\Ha^n(\psi_k\inv(K))\le C\Ha^n(K),\quad K\subset X\textrm{ compact}
\end{align}
for some $C>0$. Then $X$ has a metric fundamental class.
\end{theorem}

Note that, since $X$ is a closed manifold, the existence of $\epsilon_k$-isometries $X_k\to X$ can be upgraded to the existence of \emph{continuous} $\epsilon_k$-isometries \cite[Proposition 1.2]{ivanov97}. However, it is not clear whether these can always be chosen to have non-zero degree. Compare with \cite[Theorem 1.5]{ivanov97}.

\begin{proof}
    For each $k \in \N$, let $T_k\in \bI_n(X_k)$ be the metric fundamental class of $X_k$ satisfying
    \begin{align}\label{eq:unif-const}
    \|T_k\|\le C\Ha^n_{X_k}
    \end{align}
    for some $C>0$ independent of $k$. Let $Y_k \subset X_k$ be a maximal $\epsilon_k$-net in $X_k$, i.e.
    \[
    \bigcup_{x\in Y_k}B(x,2\epsilon_k)=X_k, \textrm{ and } B(x,\epsilon_k)\cap B(y,\epsilon_k)=\varnothing \textrm{ for distinct } x,y\in Y_k.
    \]
    Note that $\psi_k|_{Y_k}$ is $2$-Lipschitz. Embed $X \subset l^\infty$. Since $l^\infty$ is an injective metric space, there exists a $2$-Lipschitz extension $\eta_k\colon X_k \to l^\infty$ of $\psi_k$ for each $k$. It follows that $d(\psi_k,\eta_k) \leq 8\epsilon_k$. Indeed, if $x \in X_k$ then there exists $y \in Y_k$ such that $d(x,y) \leq 2\epsilon_k$. Therefore,
    \begin{equation}\label{eq: lip-mfld-approx-implies-fund-class-epsilon-isometry-approx}
        d(\psi_k(x),\eta_k(x)) \leq d(\psi_k(x),\psi_k(y)) + d(\eta_k(y), \eta_k(x)) \leq 8\epsilon_k.
    \end{equation}

Set $S_k:=\eta_{k\#}T_k$. Together \eqref{eq:vol-control} and \eqref{eq:unif-const} imply that $\sup_k\bM(S_k)<\infty$ while by \eqref{eq: lip-mfld-approx-implies-fund-class-epsilon-isometry-approx} the support of $S_k$ is contained in $\overline N_{8\epsilon_k}(X)$. (Here $N_\delta(A)$ denotes the neighborhood of $A\subset X$ in $l^\infty$; we also denote $N_\delta^X(A):=N_\delta(A)\cap X$.)  It follows from the compactness \cite[Theorem 5.2]{ambrosio-kirchheim-2000} and closure theorem \cite[Theorem 8.5]{ambrosio-kirchheim-2000} that the sequence $T_k$ converges weakly to some $T \in \bI_n(X)$ up to a subsequence. Since $\partial T_k=0$ for all $k$ we obtain that $\partial T=0$.

\medskip We claim that $T$ is a metric fundamental class of $X$. Indeed, the weak convergence $S_k\rightharpoonup T$ and Lemma \ref{lemma: integer-currents-pushfwrd-lip-constant} yield, for every compact $K\subset X$ and $\delta>0$
\begin{align*}
\|T\|(N_\delta(K))\le \liminf_{k\to\infty}\|S_k\|(N_\delta(K))\le 2^n\liminf_{k\to\infty}\|T_k\|(\eta_k\inv(N_\delta(K))),
\end{align*}
while \eqref{eq: lip-mfld-approx-implies-fund-class-epsilon-isometry-approx} implies $\eta_k\inv(N_\delta(K)) \subset \psi_k\inv(\overline N_{8\epsilon_k}^X(N_\delta^X(K)))\subset \psi_k\inv(N_{2\delta}^X(K))$ for large $k$. Together with \eqref{eq:vol-control} and \eqref{eq:unif-const} this gives
\begin{align*}
\|T\|(N_\delta(K))&\le C_1\liminf_{k\to \infty} \|T_k\|(\psi_k\inv(N_{2\delta}^X(K)))\\
&\le C_2\limsup_{k\to\infty} \Ha^n(\psi_k^{-1}(N_{2\delta}^X(K)))\le C_3\Ha^n( N_{2\delta}^X(K)),
\end{align*}
where the constants $C_1,C_2,C_3$ depend on $n$ and on the constants of \eqref{eq:vol-control} and \eqref{eq:unif-const}. Since $K$ and $\delta$ are arbitrary this implies that $\|T\|\le C_3\Ha^n$. 

\medskip Thus $T$ satisfies Definition \ref{def: metric fundamental class}(a). To prove Definition \ref{def: metric fundamental class}(b), let $M$ be the closed, oriented, smooth $n$-manifold $X$ is homeomorphic to, and let $f:X\to M$ be a Lipschitz map. Since $M$ is an absolute Lipschitz neighborhood retract, there exists a Lipschitz extension $\bar f\colon U \to M$ of $f$ to some open neighborhood $U \subset l^\infty$ of $X$. For large enough $k$ we have that $\bar f\circ \eta_k:X_k\to M$ is Lipschitz, and by the weak convergence $S_k\rightharpoonup T$ we obtain
\begin{align*}
f_\#T=\lim_k\bar f_\#S_k=\lim_k(\bar f\circ\eta_k)_\#T_k=\lim_k\deg(\bar f\circ \eta_k)\cdot \bb{M}, 
\end{align*}
where the last equality uses that $T_k$ is a metric fundamental class of $X_k$. For sufficiently large $k$, $d(\eta_k,\psi_k)$ -- and consequently $d(\bar f\circ\eta_k,f\circ\psi_k)$ -- is small enough so that $\bar f\circ\eta_k$ and $f\circ\psi_k$ are homotopic. It follows that $\deg(\bar f\circ\eta_k)=\deg(f\circ\psi_k)=\deg(f)\cdot\deg(\psi_k)$. In particular $\deg(\psi_k)$ is eventually a constant $d\in \Z\setminus\{0\}$. Thus
\[
f_\#T=d\deg(f)\bb M,
\]
and by Remark \ref{rmk:multipl} it follows that $d\inv T\in \bI_n(X)$ satisfies (a) and (b) in Definition \ref{def: metric fundamental class}.
\end{proof}

We finish the paper by tying together the results of Sections \ref{sec: degree-bound}, \ref{sec: approx-lip-mfld}, and \ref{sec:approx-lip-mfld-to-fund-class} to give a proof of Theorem \ref{thme: main}.

\begin{proof}[Proof of Theorem \ref{thme: main}]
(1)$\implies$(2) follows from Theorem \ref{thm: fund-current-implies-degree-bound}, while (2)$\implies$(3) is given by Theorem \ref{thm: lip-mfld-approx-only-degree}. The implication (3)$\implies$(1) is proved by Theorem \ref{thm: lip-mfld-approx-implies-fund-current}. Indeed, since $X_k$ is a bi-Lipschitz manifold for each $k$, it follows that
\begin{align*}
    \bM(\bb{X_k})\le n^{n/2}\Ha^n(X_k)\quad\textrm{for all }k,
\end{align*}
(see Example \ref{ex:bilip}) and we can apply Theorem \ref{thm: lip-mfld-approx-implies-fund-current} with $\psi_k:=\varphi_k\inv$. 
\end{proof}

\appendix
\section{An example}\label{sec:example}
In the following we construct a closed, oriented connected metric surface $X$ with finite Hausdorff $2$-measure and infinite Nagata dimension. Notice that such a space cannot be geodesic because, by \cite{urs-martina}, the Nagata dimension of a closed geodesic metric surface is equal to $2$. We require the following equivalent definition of Nagata dimension (cf. \cite[Proposition 2.5]{lang-nagata}):

\begin{definition}\label{def: nagata-equiv}
    The Nagata dimension of a metric space is the infimum of all integers $n$ with the following property: there exists a constant $c>0$ such that for all $s>0$ there exists a $s$-bounded cover of $X$ of the form $\mathcal{B}= \bigcup_{i=0}^n \mathcal{B}_i$ where each $\mathcal{B}_i$ is $cs$-separated.
\end{definition}

Here, a collection $\mathcal{B}$ is said to be $s$-separated for some $s>0$ if $\operatorname{dist}(A,B)>s$ for all $A,B \in \mathcal{B}$. 

\subsection*{Construction of $X$:}
Fix $n\ge 2$ for the moment. Since the Nagata dimension of $[0,1]^{n+1}$ is $n+1$, Definition \ref{def: nagata-equiv} implies that, for each $k \in \N$, there exists a $s_{k,n}>0$ such that $[0,1]^{n+1}$ does not have a $2s_{k,n}$-bounded cover of the form $\mathcal{B}= \bigcup_{i=0}^n \mathcal{B}_i$ where each $\mathcal{B}_i$ is $\frac{2}{3k}s_{k,n}$-separated.

\medskip
Let $A_{n,k} = \frac{1}{10k}s_{k,n} \Z^{n+1} \cap I^{n+1}$ and fix an ordering of $A_{n,k}$. Let $\gamma$ be a piecewise linear curve joining the points in $A_{n,k}$ in order without self-intersections. Define $X_{k,n}$ to be the union of a thin 2-dimensional cylinder around $\gamma$ and a spherical cap at one end-point of $\gamma$, endowed with the metric of $[0,1]^{n+1}$. By choosing the cylinder thin enough we can assume that it does not self-intersect and $X_{k,n}$ has area $\le 1$. Note that $X_{k,n}$ is homeomorphic to a closed disc. We obtain the surface $X$ by scaling each $X_{k,n}$ by $2^{-kn}$ and replacing small, disjoint, balls $B_{k,n}\subset \mathbb S^2$ by $2^{-nk}X_{k,n}$ glued along the boundary of $B_{k,n}$. It follows that $\Ha^2(X)\le 4\pi+\sum_{k,n}2^{-2nk}<\infty$, and that $X$ is homeomorphic to $\mathbb S^2$.

\subsection*{Nagata dimension of $X$ is infinite:} Suppose to the contrary that $\dim_X=n<\infty$ (we may assume $n\ge 2$). Then for some $k\in \N$, the property in Definition \ref{def: nagata-equiv} holds with $c=1/k$. Let $\mathcal B=\bigcup_i^n\mathcal B_i$ be a $2^{-nk}s_{n,k}$-bounded cover of $X$, with the subfamilies $\mathcal B_i$ being $2^{-nk}s_{n,k}/k$-separated. By intersecting the cover with $2^{-nk}X_{n,k}$ and rescaling by $2^{nk}$, we obtain a $2s_{k,n}$-bounded cover of $X_{k,n}$ which is the union of $n$ subfamilies each of which is $s_{k,n}/k$-separated. By abuse of notation we denote this cover and subfamilies by $\mathcal B$ and $\mathcal B_i$ ($i=1,\ldots, n)$, respectively.

\medskip Since $X_{k,n}$ is $\frac{s_{k,n}}{10k}$-dense in $[0,1]^{n+1}$ by construction, we have that if
\[
\mathcal C=\{ N_{\frac{s_{n,k}}{6k}}(B):\ B\in \mathcal B\},\quad \mathcal C_i=\{ N_{\frac{s_{n,k}}{6k}}(B):\ B\in \mathcal B_i\}\quad (i=1,\ldots,n),
\]
then $\mathcal C=\bigcup_i^n\mathcal C_i$ is a $2s_{k,n}$-bounded cover of $[0,1]^{n+1}$, and $\mathcal C_i$ is $\frac{2s_{k,n}}{3k}$-separated for each $i=1,\ldots, n$. This contradicts the choice of $s_{k,n}$ in the beginning of the construction. Thus $X$ cannot have finite Nagata dimension.

\bibliographystyle{plain}
\bibliography{refs}
\end{document}